\documentclass[11pt]{amsart}
\usepackage{amsmath,amssymb}
\usepackage{amsthm}
\usepackage{mathtools}
\usepackage{hyperref} 
\usepackage{tikz}
\usetikzlibrary{cd}
\usepackage{tikz-cd}
\usepackage{comment}
\usepackage[varg]{txfonts}

\usepackage{xcolor}
\hypersetup{
	bookmarksnumbered=true,
	colorlinks=true,
		citecolor=magenta,
		linkcolor=blue,
		urlcolor=orange,
}

\usepackage{bm}
\usepackage{mathrsfs}
\DeclareMathAlphabet{\euscr}{U}{eus}{m}{n} 
\SetMathAlphabet{\euscr}{bold}{U}{eus}{b}{n} 

\usepackage[capitalize, nameinlink]{cleveref}

\theoremstyle{definition}
\newtheorem{theorem}{Theorem}[section]
\newtheorem*{theorem*}{Theorem}
\newtheorem{definition}[theorem]{Definition}
\newtheorem*{definition*}{Definition}
\newtheorem{proposition}[theorem]{Proposition}
\newtheorem*{proposition*}{Proposition}
\newtheorem{lemma}[theorem]{Lemma}
\newtheorem*{lemma*}{Lemma}
\newtheorem{corollary}[theorem]{Corollary}
\newtheorem*{corollary*}{Corollary}
\newtheorem{remark}[theorem]{Remark}
\newtheorem*{remark*}{Remark}
\newtheorem{example}[theorem]{Example}
\newtheorem*{example*}{Example}
\newtheorem{defprop}[theorem]{Definition-Proposition}
\newtheorem*{acknowledgements*}{Acknowledgements}
\newtheorem*{outline*}{Outline of this paper}

\numberwithin{theorem}{section}

\usepackage{enumitem}

\newcommand{\Z}{\mathbb{Z}}

\newcommand{\mplaceholder}{{-}}
\newcommand{\mO}{\mathcal{O}}

\DeclareMathOperator{\Hom}{Hom}
\DeclareMathOperator{\id}{id}
\newcommand{\op}{\text{op}}

\newcommand{\yoneda}{\mathsf{y}}
\newcommand{\catname}[1]{\mathsf{#1}}
\newcommand{\Set}{\catname{Set}}

\newcommand{\Cat}{\catname{Cat}}
\newcommand{\Mod}{\catname{Mod}}
\newcommand{\Qcoh}{\catname{Qcoh}}
\newcommand{\Ab}{\catname{Ab}}
\newcommand{\Ch}{\catname{Ch}}

\newcommand{\catA}{\mathcal{A}}
\newcommand{\catB}{\mathcal{B}}
\newcommand{\catC}{\mathcal{C}}
\newcommand{\catD}{\mathcal{D}}

\newcommand{\catG}{\mathcal{G}}

\newcommand{\catJ}{\mathcal{J}}

\newcommand{\catL}{\mathcal{L}}
\newcommand{\catM}{\mathcal{M}}
\newcommand{\catS}{\mathcal{S}}
\newcommand{\catV}{\mathcal{V}}
\newcommand{\catW}{\mathcal{W}}
\newcommand{\moncatV}{\mathcal{V}}
\newcommand{\basecatV}{V}
\newcommand{\moncatW}{\mathcal{W}} 
\newcommand{\basecatW}{W}
\newcommand{\dgCat}{\catname{dgCat}}
\newcommand{\HodgCat}{\catname{HodgCat}}
\newcommand{\ChCat}{\Ch\text{-}\Cat}

\DeclareMathOperator{\obj}{ob}
\newcommand{\Psh}[1]{\catname{Psh}(#1)}
\DeclareMathOperator{\Kernel}{Ker}
\DeclareMathOperator{\Coker}{Cok}

\newcommand{\intHom}{\mathcal{H}om}
\newcommand{\tten}{\otimes^\bullet}
\newcommand{\tHom}{\Hom^\bullet}
\newcommand{\mten}{\underline{\otimes}^\bullet}
\newcommand{\mHom}{\underline{\Hom}^\bullet}

\newcommand{\VCat}{\moncatV\text{-}\Cat}
\newcommand{\WCat}{\moncatW\text{-}\Cat} 
\newcommand{\underlying}{U}
\newcommand{\Fun}{\catname{Fun}}
\newcommand{\wcolim}{\star}
\newcommand{\tensor}{\otimes}
\newcommand{\cotensor}{\pitchfork}
\DeclareMathOperator{\Lan}{Lan}
\DeclareMathOperator{\colim}{colim}

\newcommand{\fp}{\textit{fp}}
\newcommand{\qc}{\text{qc}}
\newcommand{\du}{\textit{du}}
\DeclareMathOperator{\Spec}{Spec}
\newcommand{\sheafF}{\mathcal{F}}
\newcommand{\sheafG}{\mathcal{G}}
\newcommand{\cpx}[1]{{#1}^\bullet}

\newcommand{\res}[2]{\left. #1 \right\rvert_{#2}}

\title{Grothendieck enriched categories}
\author[Y.\ Imamura]{Yuki Imamura}
\date{}
\address[Y.\ Imamura]{Department of Mathematics, Graduate School of Science, Osaka University, Machikaneyama 1-1, Toyonaka , Osaka 560-0043, Japan}
\email{u287972b@ecs.osaka-u.ac.jp}

\begin{document}

\maketitle

\begin{abstract}

In this paper, we introduce the notion of Grothendieck enriched categories for categories enriched over a sufficiently nice Grothendieck monoidal category $\mathcal{V}$, generalizing the classical notion of Grothendieck categories. 
Then we establish the Gabriel-Popescu type theorem for Grothendieck enriched categories. We also prove that the property of being Grothendieck enriched categories is preserved under the change of the base monoidal categories by a monoidal right adjoint functor.

In particular, if we take as $\mathcal{V}$ the monoidal category of complexes of abelian groups, we obtain the notion of Grothendieck dg categories. As an application of the main results, we see that the dg category of complexes of quasi-coherent sheaves on a quasi-compact and quasi-separated scheme is an example of Grothendieck dg categories.
\end{abstract}

\tableofcontents

\newpage
\section{Introduction}
\renewcommand{\thetheorem}{\Alph{theorem}}

A \emph{Grothendieck category} is a cocomplete abelian category with a generator where filtered colimits are exact. The category of modules over a ring and the category of quasi-coherent sheaves on a scheme are examples of Grothendieck categories. A Grothendieck category is known to have various good properties: it admits an injective cogenerator; any object has an injective resolution; every continuous functor from it to the category of sets is representable; the adjoint functor theorem holds.

While the definition above is given in terms of intrinsic properties of the category, a Grothendieck category is also characterized as a nice subcategory of the category of modules over a ring:

\begin{theorem*}[Gabriel-Popescu~\cite{Gabriel-Popescu:1964}; see also \cref{thm:gabriel_popescu} for a slightly generalized statement]\label{theoremA}
	Let $\catA$ be a Grothendieck category and $G\in \catA$ be a generator. Let $R$ be the ring of endomorphisms $\catA(G,G)$. Then the following assertions hold.
	\begin{enumerate}
		\item The additive functor $T=\catA(G,\mplaceholder)\colon \catA\to\Mod(R)$ has a left adjoint $S$.
		\item $T$ is fully faithful.
		\item $S$ is left exact.
	\end{enumerate}
\end{theorem*}

The Gabriel-Popescu theorem, in a word, asserts that every Grothendieck category is realized as a reflective subcategory\footnote{A subcategory is called \emph{reflective} if its inclusion functor has a left adjoint.} of the category of modules over a ring such that the left adjoint to the inclusion functor is left exact.
An important point is that the good properties of Grothendieck categories stated above can be derived from this embedding theorem.
Thus the essence of Grothendieck categories is the extrinsic characterization.
Note that a similar characterization of Grothendieck topoi as a nice subcategory is known (\cite[Cor.\ 3.5.5]{Borceux:1994HoCA3}), and Grothendieck categories can also be understood as the additive counterpart of Grothendieck topoi (\cite{Borcuex-Quinteiro:1996,Lowen:2004}).

On the other hand, the importance of dg categories has been recognized in many fields of mathematics (\cite{Bondal-Kapranov:1990, Keller:2006}). It is hence meaningful to study the analogs of Grothendieck categories and the Gabriel-Popescu theorem in the dg setting.
In fact, there are Gabriel-Popescu type theorems for triangulated categories (\cite{Porta:2010}) and for stable infinity categories (\cite{Lurie:HigherAlgebra}), both of which are relatives of dg categories. It is therefore natural to expect a similar theorem for dg categories.

Recall that dg categories are nothing but $\Ch$-enriched categories, where $\Ch$ denotes the monoidal category of cochain complexes of abelian groups.
On the other hand, an abelian category has the unique structure of a pre-additive category (a category enriched over the monoidal category $\Ab$ of abelian groups), and the Gabriel-Popescu theorem can be considered as a theorem in $\Ab$-enriched category theory.
In this paper we define the notion of a Grothendieck category enriched over $\moncatV$, where $\moncatV$ is taken from a class of monoidal categories which includes both $\Ab$ and $\Ch$, as a nice subcategory of the category of presheaves on a small $\moncatV$-category.
We then establish the Gabriel-Popescu type theorem in this generality.

We refer to a symmetric monoidal closed category that is complete and cocomplete as a \emph{cosmos}. We call a cosmos that is also a Grothendieck category a \emph{Grothendieck cosmos}.
For a Grothndieck cosmos $\moncatV=(\basecatV,\otimes,I)$, let us consider the following conditions:
\begin{enumerate}
	\item[] (C1) The unit object $I\in \moncatV$ is finitely presentable.
	\item[] (C2) $\moncatV$ has a generating set of dualizable objects.
\end{enumerate}
An object $C$ of a category $\catC$ is said to be \emph{finitely presentable} if the functor $\Hom(C,\mplaceholder)\colon \catC\to\Set$ preserves filtered colimits. An object $X$ of a symmetric monoidal closed category is said to be \emph{dualizable} if there is an isomorphism $[X,I]\otimes \mplaceholder \cong [X,\mplaceholder]$ of functors.
We will show in \cref{prop:C1_C2_implies_lfp_base} that a Grothendieck cosmos satisfying the conditions (C1) and (C2) is a \emph{locally finitely presentable base}, a monoidal category in which we can discuss the finiteness of objects. In categories enriched over such a monoidal category, finite limits make sense as in usual categories.

Then, inspired by the Gabriel-Popescu theorem, we make the following definition.

\begin{definition}[\cref{def:Grothendieck_V-category}]\label{definitionB}
	Let $\moncatV$ be a locally finitely presentable base. A $\moncatV$-category $\catA$ is said to be a \emph{Grothendieck $\moncatV$-category} if there exist a small $\moncatV$-category $\catC$ and a $\moncatV$-adjunction
	\[
	\begin{tikzpicture}
		\node at (-0.2,0) {$S:[\catC^\op,\catV]$};
		\node at (3.1,0) {$\catA : T$};
		\draw[->] (0.9,0.16) to (2.5,0.16);
		\draw[->] (2.5,-0.16) to (0.9,-0.16);
		\node at (1.7,0) {$\perp$};
	\end{tikzpicture}
	\]
	such that
	\begin{enumerate}
		\item the right adjoint $T$ is fully faithful and 
		\item the left adjoint $S$ is left exact.
	\end{enumerate}
\end{definition}

The Grothendieck cosmos $\Ch(R)$ of cochain complexes of modules over a commutative ring $R$, which we have in mind in particular for applications, satisfies the conditions (C1) and (C2). Applying \cref{definitionB} to $\moncatV=\Ch$, we get the notion of Grothendieck dg categories.

A remarkable property of Grothendieck $\moncatV$-categories is that the following adjoint functor theorem holds. This indicates the usefulness of Grothendieck $\moncatV$-categories.

\begin{proposition}[\cref{prop:AFT_for_C1C2GrothCosmos}]\label{propositionE}
	Let $\moncatV$ be a Grothendieck cosmos with the conditions (C1) and (C2) and $F\colon \catA\to\catB$ a $\moncatV$-functor.
	\begin{enumerate}
		\item Suppose $\catA$ is a Grothendieck $\moncatV$-category. If it preserves all small conical colimits, then $F$ has a right adjoint.
		\item Suppose $\catA$ is a Grothendieck $\moncatV$-category and $\catB$ is cotensored. If the underlying ordinary functor $F_0$ is cocontinuous, then $F$ has a right adjoint.
	\end{enumerate}
\end{proposition}

The definition of Grothendieck enriched categories in \cref{definitionB} is the generalization to enriched categories of the extrinsic characterization of Grothendieck categories given by the Gabriel-Popescu theorem.
Hence it is natural to ask if we can characterize Grothendieck enriched categories in terms of intrinsic properties of enriched categories in such a way that for $\Ab$-enriched categories the classical Gabriel-Popescu theorem is recovered.
The main theorem of this paper, which is stated as \cref{theoremC} below, asserts that it is in fact possible in the case that the enriching monoidal category is a nice Grothendieck cosmos.
This theorem is useful since, generally speaking, it is easier to confirm intrinsic properties of (enriched) categories rather than finding nice embeddings as in \cref{definitionB}.

\begin{theorem}[\cref{thm:main_theorem}]\label{theoremC}
	Let $\moncatV$ be a Grothendieck cosmos which satisfies the conditions (C1) and (C2). Then a $\moncatV$-category $\catA$ is a Grothendieck $\moncatV$-category if and only if it fulfills the following conditions.
	\begin{enumerate}
		\item $\catA$ is cocomplete.
		\item $\catA$ is finitely complete.
		\item $\catA$ has a $\moncatV$-generating set of objects (see \cref{def:V-generators} for the definition).
		\item The homomorphism theorem holds in $\catA$. That is, for any morphism $f$ in $\catA_0$ the canonical map $\Coker(\Kernel(f))\to\Kernel(\Coker(f))$ is an isomorphism.
		\item Conical filtered colimits are left exact. Namely for any filtered category $\catJ$, the colimit $\moncatV$-functor $\colim_\catJ \colon [\catJ_\moncatV,\catA]\to\catA$ preserves finite limits.
	\end{enumerate}
\end{theorem}

Let us give an outline of the proof of \cref{theoremC}.
It is easy to see that a Grothendieck $\moncatV$-category as defined in \cref{definitionB} satisfies the conditions (i)--(v) of \cref{theoremC}. To verify the converse, take a $\moncatV$-category $\catA$ satisfying the conditions of \cref{theoremC}. Then it is known that we can associate with the inclusion functor $\catC\hookrightarrow \catA$ an adjunction between $\catA$ and the presheaf category $[\catC^\op,\catV]$. More generally, if $\catC$ is a small $\moncatV$-category  and $\catA$ is a cocomplete $\moncatV$-category, where $\moncatV$ is a cosmos, then we can associate with any $\moncatV$-functor $F\colon \catC\to\catA$ an $\moncatV$-adjunction $\Lan_\yoneda F \dashv \Lan_F \yoneda$ of left Kan extensions as follows:
\[
\begin{tikzpicture}[auto]
	\node (hatC) at (0,0) {$[\catC^\op,\catV]$};
	\node (C) at (0,-2) {$\catC$}; \node (D) at (2.9,-2) {$\catA.$};
	
	\draw[->] (C) to node {$\scriptstyle \yoneda$} (hatC);
	\draw[->] (C) to node[swap] {$\scriptstyle F$} (D);
	\draw[->] (0.9,0) to node {$\scriptstyle \Lan_\yoneda F$}  (2.8,-1.4);
	\draw[->] (2.5,-1.7) to node {$\scriptstyle \Lan_F \yoneda$} node[sloped,pos=0.45] {$\perp$} (0.6,-0.3);
\end{tikzpicture}
\]
Here $\yoneda$ denotes the Yoneda embedding.
In this paper, we will refer to this adjunction as the \emph{nerve-and-realization ($\moncatV$-)adjunction} associated with $F$. For example, the adjunction in the Gabriel-Popescu theorem is the nerve-and-realization $\Ab$-adjunction associated with the $\Ab$-functor $R\to \catA$, where the ring $R$ is viewed as an $\Ab$-category with one single object.

Once we obtain the nerve-and-realization adjunction $S\dashv T\colon [\catC^\op,\catV]\to \catA$ associated with the inclusion functor $F\colon\catC\hookrightarrow\catA$, we only need to show that $T$ is fully faithful and that $S$ is left exact. Since $T$ is a right adjoint, it is sufficient to check that its underlying functor $T_0$ is fully faithful. Likewise, by \cref{lem:S_is_left_exact_when_S0_is_so}, it is sufficient to check that $S_0$ is left exact. Then, after we verify that $\catA_0$ is a Grothendieck category, we can use the Gabriel-Popescu theorem to prove that $T_0$ is fully faithful and $S_0$ is left exact.

As an application of \cref{theoremC}, we can easily verify that the change-of-base functor associated with a monoidal right adjoint preserves the property of being a Grothendieck enriched category.

\begin{proposition}[\cref{prop:monoidal_right_adjoint_preserve_Grothendieck_V-category}]\label{propositionD}
	Let $\moncatV,\moncatW$ be Grothendieck cosmoi and $F\dashv G\colon \moncatV\to\moncatW$ a monoidal adjunction.	Suppose that $\moncatV$ satisfies the conditions (C1) and (C2). If a $\moncatW$-category $\catB$ is a Grothendieck $\moncatW$-category, then the $\moncatV$-category $G(\catB)$ is a Grothendieck $\moncatV$-category.
\end{proposition}

As an immediate application of \cref{propositionD}, we see that for a quasi-compact and quasi-separated scheme $X$ over a commutative ring $R$, the dg category of complexes of quasi-coherent sheaves on $X$ is a Grothendieck $\Ch(R)$-category (\cref{example:dg_cat_on_qcqs_scheme_over_R}).

The Gabriel-Popescu theorem for triangulated categories is shown in \cite{Porta:2010}. It is proved that any algebraic well-generated triangulated category is a localization of the derived category of some small dg category with respect to a localizing subcategory generated by a set of objects. On the other hand, a Grothendieck dg category is pretriangulated so that its homotopy category has a natural triangulated structure. The homotopy category of a Grothendieck dg category is not well-generated, but locally well-generated in the sense of \cite{Stovicek:2010}.

However, in view of the relations of dg categories to triangulated categories and stable infinity categories, the notion of Grothendieck dg categories seems too naive. We should rather work in the localization $\HodgCat$ of $\dgCat$ with respect to quasi-equivalences and use derived dg categories in place of the dg categories of dg modules. Unfortunately, our methods do not apply to $\HodgCat$ directly (see \cref{remark:the_future_task}). It is a future task to solve this issue.

\begin{outline*}
	\cref{sec:preliminaries} is devoted to preliminaries.
	In \cref{subsec:enriched_categories} we review the basic concepts of enriched category theory, referring the reader to \cite{Kelly:1982Basic,Borceux:1994HoCA2} for further details.
	In \cref{subsec:generators_and_strong_generators} we recall the definition of a generator and a strong generator of an ordinary category and an abelian category.
	In \cref{subsec:flp_base_and_finite_limits} we give the definition of a locally finitely presentable base and describe what finite limits are like in enriched categories, following \cite{Kelly:1982Structures,Borcuex-Quinteiro-Rosicky:1998}.
	In \cref{subsec:dualizable_objects} we recall the notion of dualizable objects and the relations with weighted limits (\cref{prop:tensoring_with_dualizable_is_absolute}).
	
	In \cref{subsec:the_Gabriel-Popescu_theorem} we state the slightly generalized version of the classical Gabriel-Popescu theorem, which will be used in the proof of the main theorem (\cref{thm:main_theorem}).
	
	In \cref{subsec:Grothendieck_cosmoi} we introduce Grothendieck cosmoi satisfying the finiteness conditions (C1) and (C2). We show the fact that they are locally finitely presentable bases. This fact is due to \cite{Holm-Odabasi:2019}. We also explain their examples.
	
	Main results are given in \cref{subsec:Grothendieck_enriched_categories}. There we define the notions of enriched generators and enriched Grothendieck categories. Then we prove the Gabriel-Popescu type theorem for Grothendieck enriched categories (\cref{thm:main_theorem}). After that, as an application of the main theorem, we show in \cref{prop:monoidal_right_adjoint_preserve_Grothendieck_V-category} that the property of being Grothendieck enriched categories is preserved under changing the enrichments by a monoidal right adjoint functor. We use this proposition to give examples of Grothendieck enriched categories.
	Finally we observe that the adjoint functor theorem holds for Grothendieck enriched categories.
\end{outline*}

After the first draft of this paper appeared on the arXiv, it was pointed out by Ivan Di Liberti that \cref{def:Grothendieck_V-category} also appears in \cite{Garner-Lack:2012} under the name of \emph{$\moncatV$-topoi}.

\begin{acknowledgements*}
	The author would like to express his sincere gratitude to his supervisor, Shinnosuke Okawa, for a lot of advice and helpful inspiration through regular seminars.
	The author is also very grateful to Ivan Di Liberti for informing the author of the references on localizations and $\moncatV$-topoi (see \cref{remark:Ivan's_comments}).
\end{acknowledgements*}

\section{Preliminaries}\label{sec:preliminaries}
\setcounter{theorem}{0}%
\renewcommand{\thetheorem}{\thesection.\arabic{theorem}}

We tacitly assume that all categories discussed in this paper are locally small.

\subsection{Enriched categories}\label{subsec:enriched_categories}

The best general references here are Kelly~\cite{Kelly:1982Basic} and Borceux~\cite{Borceux:1994HoCA2}.

Let $\moncatV=(\basecatV,\otimes,I,[\mplaceholder,\mplaceholder],a, l, r, c)$ be a symmetric monoidal closed category that is complete and cocomplete, where $a$, $l$, $r$, and $c$ denote the associativity natural isomorphism $a_{XYZ}\colon (X\otimes Y)\otimes Z \xrightarrow{\cong} X\otimes (Y\otimes Z)$, the left unit natural isomorphism $l_X\colon I\otimes X\xrightarrow{\cong} X$, the right unit natural isomorphism $r_X\colon X\otimes I \xrightarrow{\cong} X$, and the symmetry natural isomorphism $c_{XY}\colon X\otimes Y \xrightarrow{\cong} Y\otimes X$, respectively. We will often omit the subscripts.

\begin{definition}
	A \emph{$\moncatV$-enriched category} $\catC$ (\emph{$\moncatV$-category} for short) consists of 
	\begin{itemize}
		\item a collection of objects $\obj(\catC)$;
		\item for every pair $C,D\in \obj(\catC)$ of objects, an object $\catC(C,D)$ of $\basecatV$, called the Hom object of $\catC$;
		\item for every triple $A,B,C\in \obj(\catC)$ of objects, the morphism $m=m_{ABC}\colon \catC(B,C)\otimes\catC(A,B) \to \catC(A,C)$ of $\basecatV$, called the composition map of $\catC$;
		\item for every object $C\in \obj(\catC)$, a morphism $j_C\colon I \to \catC(C,C)$ of $\basecatV$, called the identity map of $\catC$
	\end{itemize}
	so that the following associativity and unit diagrams commute for $A, B, C, D\in \obj(\catC)$:
	\[\begin{tikzcd}[column sep=tiny]
	{\big(\catC(C,D)\otimes \catC(B,C)\big)\otimes \catC(A,B)} \arrow{d}[swap]{m\otimes1} \arrow{rr}{a} & & {\catC(C,D)\otimes \big(\catC(B,C)\otimes \catC(A,B)\big)} \arrow{d}{1\otimes m} \\
	{\catC(B,D)\otimes \catC(A,B)} \arrow{rd}[swap]{m} && {\catC(C,D)\otimes \catC(A,C)} \arrow{ld}{m} \\
	& {\catC(A,D),} &
	\end{tikzcd}\]%
	\[\begin{tikzcd}
	{\catC(D,D)\otimes \catC(C,D)} \arrow{r}{m} & {\catC(C,D)} & {\catC(C,D)} & {\catC(C,D)\otimes \catC(C,C)} \arrow{l}[swap]{m} \\
	{I\otimes \catC(C,D),} \arrow{u}{j\otimes 1} \arrow{ru}[swap]{l} &&& {\catC(C,D)\otimes I.} \arrow{u}[swap]{1\otimes j} \arrow{ul}{r}
	\end{tikzcd}\]
	
	We will write $C\in \catC$ to indicate that $C$ is an object of $\catC$.
\end{definition}

\begin{definition}	
	For $\moncatV$-categories $\catC$ and $\catD$, a \emph{$\moncatV$-functor} $F\colon \catC \to \catD$ consists of
	\begin{itemize}
		\item a function $F\colon \obj\catC\to\obj\catD, C \mapsto FC$;
		\item for every pair $C,D\in \catC$ of objects, a morphism $F_{CD}\colon \catC(C,D)\to\catD(FC,FD)$ of $\basecatV$
	\end{itemize}
	so that the following diagrams commute for $A,B, C\in \catC$:
	\[\begin{tikzcd}
	{\catC(B,C)\otimes\catC(A,B)} \arrow{r}{m} \arrow{d}[swap]{F_{BC}\otimes F_{AB}} & {\catC(A,C)} \arrow{d}{F_{AC}} \\
	{\catD(FB,FC)\otimes \catD(FA,FB)} \arrow{r}[swap]{m} & {\catD(FA,FC),}
	\end{tikzcd}\]%
	\[\begin{tikzcd}[row sep=tiny]
	& {\catC(C,C)} \arrow{dd}{F_{CC}} \\
	{I} \arrow{ru}{j_C} \arrow{rd}[swap]{j_{FC}} & \\
	& {\catD(FC,FC).}
	\end{tikzcd}\]
\end{definition}

\begin{definition}
	For $\moncatV$-functors $F,G\colon \catC \to \catD$, a \emph{$\moncatV$-natural transformation} $\alpha\colon F\to G$ consists of 
	\begin{itemize}
		\item for every object $C\in \catC$, a morphism $\alpha_C\colon I \to \catD(FC,GC)$ of $\basecatV$
	\end{itemize}
	so that the following $\moncatV$-naturality diagram commutes for $C,D\in \catC$:
	\[\begin{tikzcd}
	& {I\otimes\catC(C,D)} \arrow{r}{\alpha_B\otimes F} & {\catD(FD,GD)\otimes \catD(FC,FD)} \arrow{rd}{m} & \\
	{\catC(C,D)} \arrow{ru}{l^{-1}} \arrow{rd}[swap]{r^{-1}} &&& {\catD(FC,GD).} \\
	& {\catC(C,D)\otimes I}\arrow{r}[swap]{G\otimes \alpha_C} & {\catD(GC,GD)\otimes \catD(FC,GC)} \arrow{ru}[swap]{m} &
	\end{tikzcd}\]
\end{definition}

Ordinary categories are $\Set$-enriched categories, where $\Set$ denotes the cartesian category of sets; preadditive categories are $\Ab$-enriched categories, where $\Ab$ denotes the monoidal category of abelian groups; dg categories are $\Ch$-enriched categories, where $\Ch$ denotes the category of complexes of abelian groups.

The monoidal category $\moncatV$ has a natural $\moncatV$-category structure, together with the internal-hom $[\mplaceholder,\mplaceholder]$ of $\basecatV$. By abuse of notation, we will write $\catV$ for this $\moncatV$-category. For a $\moncatV$-category $\catC$, we obtain the \emph{opposite $\moncatV$-category} $\catC^\op$ by setting $\catC^\op(C,D)\coloneqq \catC(D,C)$. We also obtain the \emph{underlying category} $\catC_0$ as the ordinary category whose objects are the same as those of $\catC$ and whose Hom sets are $Hom_\basecatV(I,\catC(C,D))$. Note $(\catV)_0\cong \basecatV$.

$\moncatV$-categories, $\moncatV$-functors, and $\moncatV$-natural transformations constitute a $2$-category $\VCat$. Let $\underlying=(\mplaceholder)_0\colon \VCat\to\Cat$ be the functor taking underlying categories of enriched categories. The functor $(\mplaceholder)_0$ has a left adjoint $(\mplaceholder)_\moncatV\colon \Cat\to\VCat$ and for an ordinary category $\catL$, the $\moncatV$-category $\catL_\moncatV$ is called the \emph{free $\moncatV$-category} on $\catL$.

Two objects $C$, $D$ of a $\moncatV$-category $\catC$ are said to be \emph{isomorphic} if they are isomorphic in the underlying category $\catC_0$; then we write $C\cong D$. We say that a $\moncatV$-functor $F\colon \catC \to \catD$ is \emph{fully faithful} if all of the maps $F_{CD}\colon \catC(C,D)\to \catD(FC,FD)$ between Hom objects are isomorphisms, and that $F$ is \emph{essentially surjective} if for every $D\in \catD$, there is an object $C\in \catC$ such that $F(C)\cong D$. A $\moncatV$-functor is called an \emph{equivalence} if it is fully faithful and essentially surjective.

For $\moncatV$-categories $\catC$ and $\catD$, let $\Fun(\catC,\catD)$ denote the ordinary category of $\moncatV$-functors $\catC\to\catD$ and $\moncatV$-natural transformations. If $F,G\colon \catC \to \catD$ are $\moncatV$-functors, write the morphisms that correspond under the tensor--internal-hom adjunction to the composites
\begin{align*}
\catD(FC,GC)\otimes\catC(C,D) \xrightarrow{1\otimes G_{CD}} \catD(FC,GC)\otimes \catD(GC,GD) & \\
\xrightarrow{\makebox[1em]{$\scriptstyle c$}} \catD(GC,GD)\otimes \catD(FC,GC) 
&\xrightarrow{m} \catD(FC,GD), \\
\catD(FD,GD)\otimes \catC(C,D) \xrightarrow{1\otimes F_{CD}} \catD(FD,GD)\otimes \catD(FC,FD) &\xrightarrow{m} \catD(FC,GD)
\end{align*}
as
\begin{align*}
\tau_{CD} &\colon \catD(FC,GC)\to [\catC(C,D),\catD(FC,GD)], \\
\xi_{CD} &\colon \catD(FD,GD)\to [\catC(C,D),\catD(FC,GD)], 
\end{align*}
respectively. Define $[\catC,\catD](F,G) \in \basecatV$ by the following equalizer diagram:
\[
\begin{tikzcd}
{[\catC,\catD]}(F,G) \arrow{r} & \prod_{C\in\catC} \catD(FC,GC) \arrow[shift left=0.7ex]{r}{\prod \tau_{CD}} \arrow[shift right=0.7ex]{r}[swap]{\prod \xi_{CD}} & \prod_{C,D\in \catC} {[\catC(C,D),\catD(FC,GD)].}
\end{tikzcd}
\]
Then we have the \emph{functor enriched category} $[\catC,\catD]$ of enriched functors with its Hom object $[\catC,\catD](F,G)$; its underlying category $[\catC,\catD]_0$ is isomorphic to $ \Fun(\catC,\catD)$. If $\obj\catC$ is a set, then $\moncatV$-category $\catC$ is said to be \emph{small}. If $\catC$ is small, the equalizer $[\catC,\catD](F,G)$ exists, which implies the existence of the functor enriched category $[\catC,\catD]$. We also have the Yoneda embedding $\yoneda\colon \catC\to[\catC^\op,\catV]$, which is a $\moncatV$-functor that sends each object $C\in \catC$ to the representable $\moncatV$-functor $\catC(\mplaceholder,C)$.

As in ordinary category theory, the Yoneda lemma holds for enriched categories.

\begin{theorem}[{The enriched Yoneda lemma~\cite[(2.31)]{Kelly:1982Basic}}]\label{thm:enriched_Yoneda_lemma}
	For a $\moncatV$-functor $F\colon \catC \to \catV$ and an object $C\in \catC$, there is an isomorphism
	\[ [\catC,\catV](\catC(C,\mplaceholder),F) \cong FC \]
	which is natural in $F$ and $C$.
\end{theorem}

Next we recap the notion of limits for enriched categories. It is wider than that of limits in ordinary category theory in that ordinary limits correspond to the so-called conical limits, a particular kind of limits in enriched categories.

\begin{definition}
	Let $\catJ$, $\catC$ be $\moncatV$-categories and $F\colon\catJ \to \catC$ be a $\moncatV$-functor. For a $\moncatV$-functor $W\colon \catJ\to\catV$, a \emph{limit of $F$ weighted by $W$} is defined as an object $\{W,F\}\in \catC$ together with an isomorphism
	\[ \catC(C,\{W,F\}) \cong [\catJ,\catV](W,\catC(C,F\mplaceholder)) \]
	natural in $C\in \catC$.
	
	For a $\moncatV$-functor $W\colon \catJ^\op\to\catV$, a \emph{colimit of $F$ weighted by $W$} is defined as an object $W\wcolim F\in \catC$ together with an isomorphism
	\[ \catC(W\wcolim F,C) \cong [\catJ^\op,\catV](W,\catC(F\mplaceholder,C)) \]
	natural in $C\in \catC$.
	
	Here $W$ is called a \emph{weight} of the limit or colimit; if $\catJ$ is small, $W$ is said to be \emph{small}. We call a $\moncatV$-category $\catC$ \emph{complete} if it admits all small limits, and \emph{cocomplete} if it admits all small colimits.
\end{definition}

Note that the $\moncatV$-category $\catV$ is complete and cocomplete, and so is the presheaf $\moncatV$-category $[\catC^\op,\catV]$ on a small $\moncatV$-category $\catC$.

As a special case of weighted limits, we can define a limit for an ordinary functor $H\colon\catL\to\catC_0$ from an ordinary category to the underlying category of a $\moncatV$-enriched category. Recall that the functor $(\mplaceholder)_0$ that takes underlying categories has the left adjoint $(\mplaceholder)_\moncatV$. Under this adjunction, $H$ corresponds to the $\moncatV$-functor $\widetilde{H}\colon \catL_\moncatV\to\catC$ and the constant functor $\Delta I\colon \catL \to \catV_0$ at the unit object $I$ corresponds to the $\moncatV$-functor $\widetilde{\Delta I}\colon \catL_\moncatV\to\catV$. Then we define a \emph{conical limit of $H$} as the limit $\{\widetilde{\Delta I}, \widetilde{H}\}$ of $\widetilde{H}$ weighted by $\widetilde{\Delta I}$. If the conical limit of $H$ exists in $\catC$, then it becomes an ordinary limit $\lim H$ of $H\colon \catL \to \catC_0$ in $\catC_0$ (see \cite[(3.53)]{Kelly:1982Basic}). Hence the completeness of $\catC$ implies that of the ordinary category $\catC_0$. Dually, we obtain the notion of conical colimits.

\begin{definition}
	Let $\catC$ be a $\moncatV$-category. For $X\in \catV$ and $D\in \catC$, a \emph{cotensor product} of $X$ and $D$ is defined as an object $X\cotensor D\in \catC$ together with an isomorphism
	\[ \catC(C, X\cotensor D) \cong [X,\catC(C,D)] \]
	natural in $C\in \catC$. This is a special case of weighted limits.
	
	A \emph{tensor product} of $X$ and $D$ is defined as an object $X\tensor D\in \catC$ together with an isomorphism
	\[ \catC(X\tensor D,C) \cong [X,\catC(D,C)] \]
	natural in $C\in \catC$. This is a special case of weighted colimits.
	
	A $\moncatV$-category $\catC$ is called \emph{cotensored} if it admits all cotensor products, and \emph{tensored} if it admits all tensored products.
\end{definition}

For example, since we have $[C,[X,D]]\cong [X,[C,D]]$ in the $\moncatV$-category $\catV$ by the tensor--internal-hom adjunction, the cotensor product of objects $X$ and $D$ of $\catV$ is the internal-hom $[X,D]$. In a similar manner, we see that the tensor product of $X$ and $D$ in $\catV$ is the monoidal product $X\otimes D$.

\begin{definition}
	Let $\catC$ be a $\moncatV$-category and $\Hom_\catC\colon \catC^\op\otimes\catC\to \catV$ denote its Hom $\moncatV$-functor. For a $\moncatV$-functor $G\colon \catC^\op\otimes\catC\to\catD$, an \emph{end} of $G$ is defined to be the limit of $G$ weighted by $\Hom_\catC$, written as
	\[ \int_{C\in \catC} G(C,C) \coloneqq \{\Hom_\catC,G\}. \]
	
	A \emph{coend} of $G$ is defined to be the colimit of $G$ weighted by $\Hom_\catC$, written as 
	\[ \int^{C\in \catC} G(C,C) \coloneqq \Hom_\catC \wcolim G. \]
\end{definition}

For example, the Hom object $[\catC,\catD](F,G)$ between $\moncatV$-functors $F,G\colon \catC \to \catD$ is the end of the $\moncatV$-functor $\catD(F\mplaceholder,G\mplaceholder)\colon \catC^\op \otimes \catC \to \catV$ (see \cite[(2.10)]{Kelly:1982Basic}): that is,
\[ [\catC,\catD](F,G) \cong \int_{C\in \catC} \catD(FC,GC). \]

We have introduced three particular kinds of weighted limits: conical limits, cotensor products, and ends. We next explain the relationship among them.

\begin{proposition}[{\cite[\S3.8]{Kelly:1982Basic}}]\label{prop:tensored_and_ordinary_limit_implies_conical_limit}
	Let $\catC$ be a $\moncatV$-category and $H\colon \catL\to\catC_0$ be an ordinary functor from an ordinary category $\catL$. Suppose $\catC$ is tensored. Then the conical limit $\{\Delta I, \widetilde{H}\}$ of $H$ exists if and only if the ordinary limit $\lim H$ of $H$ exists in $\catC_0$.
	
	Dually, provided that $\catC$ is cotensored, the conical colimit of $H$ exists if and only if the ordinary colimit $\colim H$ of $H$ exists.
\end{proposition}

\begin{proposition}[{\cite[\S3.10]{Kelly:1982Basic}}]\label{prop:end_is_conical_limit_of_cotensor}
	Let $\catC$, $\catD$ be $\moncatV$-categories and $G\colon \catC^\op\otimes\catC\to\catD$ be a $\moncatV$-functor. Suppose that $\catC$ is small and that $\catD$ is cotensored and admits all small conical limits. Then the end $\int_{C\in \catC} G(C,C)$ of $G$ exists.
	
	Dually, provided that $\catD$ is tensored and admits all small conical colimits, then the coend $\int^{C\in \catC} G(C,C)$ of $G$ exists.
\end{proposition}

\begin{proposition}[{\cite[(3.69), (3.70)]{Kelly:1982Basic}}]\label{prop:limit_is_end_of_cotensor}
	Let $\catJ$, $\catC$ be $\moncatV$-categories and $F\colon\catJ \to \catC$ be a $\moncatV$-functor.
	\begin{enumerate}
		\item If $\catC$ is cotensored, then for any $\moncatV$-functor $W\colon \catJ \to \catV$, we have
		\[ \{W,F\} \cong \int_{J\in \catJ} WJ \cotensor FJ, \]
		either side existing if the other does.
		
		\item If $\catC$ is tensored, then for any $\moncatV$-functor $W\colon \catJ^\op \to \catV$, we have
		\[ W\wcolim F \cong \int^{J\in \catJ} WJ \tensor FJ, \]
		either side existing if the other does.
	\end{enumerate}
\end{proposition}

\begin{theorem}[{\cite[Thm.\ 3.73]{Kelly:1982Basic}}]\label{thm:complete=cotensored_and_small_conical_limit}
	A $\moncatV$-category $\catC$ is complete if and only if it is cotensored and admits all small conical limits. 
	
	Daully, $\catC$ is cocomplete if and only if it is tensored and admits all small conical colimits. 
\end{theorem}

\begin{corollary}\label{cor:C_0_complete_implies_C_complete}
	A cotensored and tensored $\moncatV$-category $\catC$ is complete (or cocomplete) if $\catC_0$ is so.
\end{corollary}

\begin{proof}
	If $\catC_0$ is complete, then by \cref{prop:tensored_and_ordinary_limit_implies_conical_limit} $\catC$ admits all small conical limits. Thus \cref{thm:complete=cotensored_and_small_conical_limit} shows that $\catC$ is complete as an enriched category.
\end{proof}

\begin{proposition}[{\cite[(3.71), (3.72)]{Kelly:1982Basic}}]\label{prop:coYoneda_lemma}
	Let $\catJ$, $\catC$ be $\moncatV$-categories and $F\colon \catJ\to\catC$ be a $\moncatV$-functor. If the following limits or colimits exist, then we have
	\begin{align*}
		\int_{A\in \catJ}\catJ(J,A) \cotensor FA &\cong FJ, \\
		\int^{A\in \catJ} \catJ(A,J) \tensor FA &\cong FJ.
	\end{align*}
\end{proposition}

We can also consider adjunctions and Kan extensions in enriched category theory.

\begin{definition}
	Let $F\colon \catC\to \catD$ and $G\colon \catD \to \catC$ be $\moncatV$-functors. The pair $(F,G)$ is called a \emph{$\moncatV$-adjunction} if there is an isomorphism
	\[ \catD(FC,D) \cong \catC(C,GD) \]
	natural in $C\in \catC$ and $D \in \catD$. If $(F,G)$ is a $\moncatV$-adjunction, then we write $F\dashv G$ and call $F$ a \emph{left adjoint} and $G$ a \emph{right adjoint}.
\end{definition}

\begin{proposition}[{\cite[\S1.11]{Kelly:1982Basic}}]\label{prop:right_adjoint_is_fullyfaithful_iff_its_underlying_is_so}
	A right adjoint is fully faithful if and only if its underlying functor is so.
\end{proposition}

\begin{definition}\label{def:Kan_extension}
	Let $F\colon \catC\to\catM$ and $K\colon\catC\to\catD$ be $\moncatV$-functors. A $\moncatV$-functor $\Lan_K F\colon \catD\to\catM$ is called a \emph{left Kan extension of $F$ along $K$} if for any $\moncatV$-functor $S\colon \catD\to\catM$, there is a natural bijection
	\[ \Hom_{\Fun(\catD,\catM)}(\Lan_K F,S) \cong \Hom_{\Fun(\catC,\catM)}(F,SK). \]
\end{definition}

\begin{definition}\label{def:pointwise_Kan_extension}
	Let $F\colon \catC\to\catM$ and $K\colon\catC\to\catD$ be $\moncatV$-functors. A $\moncatV$-functor $T\colon\catD\to\catM$ is called a \emph{pointwise left Kan extension of $F$ along $K$} if for any $d\in\catD$ and $m \in \catM$, there is a natural isomorphism
	\[ \catM(Td,m) \cong [\catC^\op,\catV](\catD(K\mplaceholder,d),\catM(F\mplaceholder,m)). \]
\end{definition}

Note that in Kelly's book~\cite{Kelly:1982Basic}, Kan extensions in the sense of \cref{def:Kan_extension} are referred to as weak Kan extensions and pointwise Kan extensions in the sense of \cref{def:pointwise_Kan_extension} as strong Kan extensions.

\begin{theorem}[{\cite[Thm.\ 4.43]{Kelly:1982Basic}}]
	Pointwise left Kan extensions are left Kan extensions.
\end{theorem}

\begin{proposition}[{\cite[(4.25)]{Kelly:1982Basic}}]\label{prop:exist_of_pointwise_Kan_extension}
	Let $F\colon \catC\to\catM$ and $K\colon\catC\to\catD$ be $\moncatV$-functors. If $\catM$ is cocomplete, then the pointwise left Kan extension $\Lan_K F$ exists and for any $d\in \catD$ it holds that
	\[ \Lan_K F(d)= \int^{c\in \catC} \catD(Kc,d)\tensor Fc. \]
\end{proposition}

\begin{proposition}\label{prop:Lan_of_yoneda}
	Let $\catC$ be a small $\moncatV$-category and $F\colon \catC \to\catD$ a $\moncatV$-functor. Then the left Kan extension $\Lan_F \yoneda\colon \catD\to[\catC^\op,\catV]$ of $F$ along the Yoneda embedding $\yoneda\colon \catC\to[\catC^\op,\catV]$ satisfies $\Lan_F \yoneda(d)\cong\catD(F\mplaceholder,d)$
	for any $d\in \catD$.
\end{proposition}

\begin{proof}
	Since the presheaf $\moncatV$-category $[\catC^\op,\catV]$ is cocomplete, the pointwise left Kan extension $\Lan_F \yoneda$ exists. Then using the enriched Yoneda lemma (\cref{thm:enriched_Yoneda_lemma}), we have natural isomorphisms
	\begin{alignat*}{2}
	[\catC^\op,\catV](\Lan_F \yoneda(d),P) &\cong [\catC^\op,\catV](\catD(F\mplaceholder,d),[\catC^\op,\catV](\yoneda\mplaceholder,P))  \\
	&\cong [\catC^\op,\catV](\catD(F\mplaceholder,d),P) 
	\end{alignat*}
	for $P\in [\catC^\op,\catV]$. Thus the Yoneda lemma implies $\Lan_F \yoneda(d)\cong\catD(F\mplaceholder,d)$.
\end{proof}

\begin{theorem}\label{thm:ubiquitus_adjunction}
	Let $\catC$ be a small $\moncatV$-category and $F\colon \catC \to\catD$ a $\moncatV$-functor.  If the pointwise left Kan extension $\Lan_\yoneda F$ exists, there is a $\moncatV$-adjunction $\Lan_\yoneda F \dashv \Lan_F \yoneda$.
	\[
	\begin{tikzpicture}[auto]
	\node (hatC) at (0,0) {$[\catC^\op,\catV]$};
	\node (C) at (0,-2) {$\catC$}; \node (D) at (2.8,-2) {$\>\catD.$};
	
	\draw[->] (C) to node {$\scriptstyle \yoneda$} (hatC);
	\draw[->] (C) to node[swap] {$\scriptstyle F$} (D);
	\draw[->] (0.9,0) to node {$\scriptstyle \Lan_\yoneda F$}  (2.8,-1.4);
	\draw[->] (2.5,-1.7) to node {$\scriptstyle \Lan_F \yoneda$} node[sloped,pos=0.45] {$\perp$} (0.6,-0.3);
	\end{tikzpicture}
	\]
	In particular, this is the case if $\catD$ is cocomplete, because of \cref{prop:exist_of_pointwise_Kan_extension}.
\end{theorem}

\begin{proof}
	By the Yoneda lemma (\cref{thm:enriched_Yoneda_lemma}) and \cref{prop:Lan_of_yoneda}, for any $d\in \catD$ and $P\in [\catC^\op,\catV]$ we have natural isomorphisms
	\begin{alignat*}{2}
	\catD(\Lan_\yoneda F(P),d) &\cong [\catC^\op,\catV]([\catC^\op,\catV](\yoneda\mplaceholder,P),\catD(F\mplaceholder,d)) \\
	&\cong [\catC^\op,\catV](P,\catD(F\mplaceholder,d)) \\
	&\cong [\catC^\op,\catV](P,\Lan_F \yoneda(d)). 
	\end{alignat*}
	Hence $\Lan_\yoneda F$ is a left adjoint to $\Lan_F \yoneda$.
\end{proof}

In this paper, we will calll the $\moncatV$-adjunction of \cref{thm:ubiquitus_adjunction} the \emph{nerve-and-realization adjunction} associated with $F$, inspired by \cite[``\href{https://ncatlab.org/nlab/show/nerve+and+realization}{nerve and realization}'']{nLab}, though it is somewhat lengthy.

\begin{proposition}[{\cite[Prop.\ 4.23]{Kelly:1982Basic}}]\label{prop:Kan_extension_along_full_faithful}
	Let $F\colon \catC\to\catM$ and $K\colon\catC\to\catD$ be $\moncatV$-functors. Suppose that the pointwise left Kan extension $\Lan_K F$ of $F$ along $K$ exists. If $K$ is fully faithful, then we have $\Lan_K F \circ K \cong F$.
\end{proposition}

\begin{theorem}[{\cite[Thm.\ 4.51]{Kelly:1982Basic}}]\label{thm:adjunction_from_functor_cat}
	Let $\catC$ be a small $\moncatV$-category and $\catD$ be a cocomplete $\moncatV$-category. Then there is a one-to-one correspondance between the following sets:
	\begin{enumerate}
		\item the set of isomorphic classes of $\moncatV$-functors $F\colon\catC\to\catD$,
		\item the set of isomorphic classes of cocontinuous $\moncatV$-functors $S\colon [\catC^\op,\catV]\to\catD$,
		\item the set of isomorphic classes of $\moncatV$-adjunctions $S\dashv T\colon [\catC^\op,\catV]\to\catD$.
	\end{enumerate}
\end{theorem}

\begin{proof}
	The correspondance between the first set and the second is given by $F\mapsto \Lan_\yoneda F$ and $S\mapsto S\circ\yoneda$. The correspondance between the second set and the third is given by $S\mapsto (S\dashv \Lan_{S\yoneda} \yoneda)$ and $(S\dashv T)\mapsto S$. Details are left to the reader.
\end{proof}

When we have two monoidal categories and a functor between them that preserves monoidal structures in a sense, we can change enrichments of enriched categories.

\begin{definition}
	Let $\moncatV=(\basecatV,\otimes,I)$ and $\moncatW=(\basecatW,\otimes,I')$ be monoidal categories.
	A \emph{lax monoidal functor} bwtween them consists of 
	\begin{itemize}
		\item a functor $F\colon \moncatV\to\moncatW$;
		\item for every pair $X,Y\in \moncatV$ of objects, a natural morphism $\tau_{XY}\colon F(X)\otimes F(Y)\to F(X\otimes Y)$ of $\moncatW$;
		\item a morphism $\sigma\colon I'\to F(I)$ of $\moncatW$
	\end{itemize}
	so that these morphisms are compatible with the monoidal structures (see \cite[Def.\ 6.4.1]{Borceux:1994HoCA2}). If $\tau_{XY}$ and $\sigma$ all are isomorphisms, then $F$ is said to be \emph{strongly monoidal}.
\end{definition}

Note that the representable functor $\Hom_\basecatV(I,\mplaceholder)\colon \moncatV\to\Set$ is lax monoidal.

For an adjunction $F\dashv G\colon \moncatV\to\moncatW$ between monoidal categories, the right adjoint $G$ is lax monoidal if the left adjoint $F$ is strongly monoidal (see \cite[``\href{https://ncatlab.org/nlab/show/monoidal+adjunction}{monoidal adjunction}'']{nLab}). We refer to an adjunction whose left adjoint is strongly monoidal as a \emph{monoidal adjunction}.

\begin{proposition}[{Change of base \cite[Prop.\ 6.4.3]{Borceux:1994HoCA2}}]\label{prop:change_of_base_2functor}
	Let $F\colon \moncatV\to\moncatW$ be a lax monoidal functor between monoidal categories. Then $F$ induces a functor $F\colon \VCat\to\WCat$ between categories of enriched categories, which sends a $\moncatV$-category $\catC$ to the $\moncatW$-category $F(\catC)$ such that
	\begin{itemize}
		\item its objects are the same as those of $\catC$;
		\item its Hom objects are $F(\catC(C,D))\in \moncatW$;
		\item its composition maps are 
		\[ F(\catC(B,C))\otimes F(\catC(A,B)) \xrightarrow{\tau} F(\catC(B,C)\otimes\catC(A,B)) \xrightarrow{F(m)} F(\catC(A,C)); \]
		\item its identity maps are
		\[ I' \xrightarrow{\sigma} F(I) \xrightarrow{F(j_C)} F(\catC(C,C)). \]
	\end{itemize}
\end{proposition}

If we take the lax monoidal functor $\Hom_\basecatV(I,\mplaceholder)\colon \moncatV\to\Set$ for $F$ in \cref{prop:change_of_base_2functor}, then we get the functor $(\mplaceholder)_0\colon \VCat\to\Cat$ taking underlying categories.

Now we consider a lax monoidal functor that is the right adjoint of a monoidal adjunction.

\begin{proposition}\label{prop:monoidal_right_ajoint_commute_with_underlying_functor}
	Let $F\dashv G\colon \moncatV\to\moncatW$ be a monoidal adjunction. Then for any $\moncatW$-category $\catD$, there is an isomorphism $(G(\catD))_0\cong \catD_0$ of ordinary categories.
\end{proposition}

\begin{proof}
	Since $F$ is strongly monoidal, we have
	\[ \Hom_\basecatV(I,G\mplaceholder) \cong \Hom_\basecatW(F(I),\mplaceholder) \cong \Hom_\basecatW(I',\mplaceholder), \]
	which induces the isomorphism.
\end{proof}

\begin{proposition}\label{prop:enriched_ver_of_monoidal_adjunction}
	Let $F\dashv G\colon \moncatV\to\moncatW$ be a monoidal adjunction. Then for $X \in \moncatV$ and $Y \in \moncatW$, there is a natural isomorphism
	\[ G(\catW(F(X),Y)) \cong \catV(X,G(Y)). \]
\end{proposition}

\begin{proof}
	For any $Z\in \catV$, we have
	\begin{align*}
	\Hom_\basecatV(Z, G(\catW(F(X),Y)))
	&\cong \Hom_\basecatW(F(Z),\catW(F(X),Y)) \\
	&\cong \Hom_\basecatW(F(Z)\otimes F(X),Y) \\
	&\cong \Hom_\basecatW(F(Z\otimes X),Y) \\
	&\cong \Hom_\basecatV(Z\otimes X,G(Y)) \\
	&\cong \Hom_\basecatV(Z, \catV(X,G(Y))).
	\end{align*}
	Hence we obtain
	\[ G(\catW(F(X),Y)) \cong \catV(X,G(Y)) \]
	by the Yoneda lemma.
\end{proof}

\cref{prop:enriched_ver_of_monoidal_adjunction} states that a monoidal adjunction $F\dashv G$ becomes a $\moncatV$-adjunction between $\catV$ and $G(\catW)$.

\begin{proposition}\label{prop:monoidal_right_adjoint_preserve_cotensoredness}
	Let $F\dashv G\colon \moncatV\to\moncatW$ be a monoidal adjunction. If a $\moncatW$-category $\catD$ is cotensored (or tensored), then the $\moncatV$-category $G(\catD)$ is so.
\end{proposition}

\begin{proof}
	In case that $\catD$ is tensored, for $X\in \catV$ and $C,D\in G(\catD)$ we have
	\begin{align*}
	\catV(X, G(\catD)(C,D))
	&= \catV(X, G(\catD(C,D))) \\
	&\cong G(\catW(F(X),\catD(C,D))) \\
	&\cong G(\catD(F(X)\tensor_\moncatW C, D)) \\
	&= G(\catD)(F(X)\tensor_\moncatW C,D)
	\end{align*}
	by \cref{prop:enriched_ver_of_monoidal_adjunction}. Therefore the $\moncatV$-category $G(\catD)$ admits the tensor product as $X \tensor_\moncatV C = F(X)\tensor_\moncatW C$.
	
	In a similar way, when $\catD$ is cotensored, we have
	\begin{align*}
	\catV(X, G(\catD)(C,D))
	&= \catV(X, G(\catD(C,D))) \\
	&\cong G(\catW(F(X),\catD(C,D))) \\
	&\cong G(\catD(C, F(X)\cotensor_\moncatW D)) \\
	&= G(\catD)(C,F(X)\cotensor_\moncatW D).
	\end{align*}
	Therefore $G(\catD)$ admits the cotensor product as $X \cotensor_\moncatV C = F(X)\cotensor_\moncatW C$.
\end{proof}

\begin{corollary}\label{cor:monoidal_right_adjoint_preserve_completeness}
	Let $F\dashv G\colon \moncatV\to\moncatW$ be a monoidal adjunction. If a $\moncatW$-category $\catD$ is complete and cocomplete, then the $\moncatV$-category $G(\catD)$ is so.
\end{corollary}

\begin{proof}
	If a $\moncatW$-category $\catD$ is complete and cocomplete, then the underlying category $\catD_0$, and hence $(G(\catD))_0$ by \cref{prop:monoidal_right_ajoint_commute_with_underlying_functor}, is complete and cocomplete as an ordinary category. On the other hand, \cref{prop:monoidal_right_adjoint_preserve_cotensoredness} shows that $G(\catD)$ is cotensored and tensored. Therefore it follows from \cref{cor:C_0_complete_implies_C_complete} that $G(\catD)$ is complete and cocomplete as an enriched category.
\end{proof}

\subsection{Generators and strong generators of categories}\label{subsec:generators_and_strong_generators}

Write $\Psh{\catS}=[\catS^\op,\Set]$ for the presheaf category  on a small ordinary category $\catS$.

We will refer to a nonempty set of objects of a category $\catC$ as \emph{a family of objects} $S\subseteq \obj(\catC)$. We identify a family of objects $S$ with a small full subcategory $\catS\subseteq \catC$ spanned by $S$.

Recall that a functor $F\colon \catC\to\catD$ is said to be \emph{faithful} if for any pair of morphisms $f,g\colon c\to c'$ of $\catC$, we have $f=g$ whenever $F(f)=F(g)$.

\begin{definition}
	A family of functors $\{F_i\colon \catC \to \catD\}_{i \in I}$ is said to be \emph{jointly faithful} if for any pair of morphisms $f,g\colon c\to c'$ of $\catC$, we have $f=g$ whenever $F_i(f)=F_i(g)$ for all $i \in I$.
\end{definition}

If $\{F_i\}_i$ has only one element $F$, then that $\{F_i\}_i=\{F\}$ is jointly faithful just means that $F$ is faithful.

\begin{defprop}\label{defprop:generator}
	For a family of objects $S$ of a category $\catC$, the following conditions are equivalent. We call $S$ a \emph{generating set of objects} if it satisfies one of them (hence all of them).
	\begin{enumerate}
		\item For any pair $g_1,g_2\colon c\to d$ of morphisms of $\catC$, we obtain $g_1=g_2$ if $g_1\circ f=g_2\circ f$ holds for all $s\in S$ and $f\colon s\to c$ in $\catC$.
		\item The family of functors $\{\Hom(s,\mplaceholder)\}_{s\in S}$ is jointly faithful.
		\item The left Kan extension $\Lan_F \yoneda \colon \catC\to \Psh{\catS};\, c\mapsto \res{\Hom(\mplaceholder,c)}{\catS^\op}$ is faithful, where $F\colon \catS \hookrightarrow \catC$ denotes the inclusion functor.
	\end{enumerate}
	If moreover $\catC$ admits coproducts, then the following conditions are equivalent to the above.
	\begin{enumerate}[resume]
		\item For any $c\in \catC$, the induced map $\displaystyle \gamma_c\colon \coprod_{s\in S,\, f\colon s\to c}s\to c$ is epimorphic.
		\item For any $c\in \catC$, there is a family of morphisms $\{f_i\colon s_i \to c\}_{i \in I}$ of $\catC$ such that each $s_i$ is in $S$ and the induced map $\coprod_i s_i \to c$ is epimorphic.
	\end{enumerate}
\end{defprop}

\begin{proof}
	It is easy to check.
\end{proof}

We call an object $G$ a \emph{generator} if $S=\{G\}$ is a generating set of objects.

\begin{definition}
	A functor $F\colon \catC\to\catD$ is said to be \emph{conservative} if it reflects isomorphisms, that is, a morphism $f\colon c\to c'$ in $\catC$ is an isomorphism whenever $F(f)$ is so. More generally, a family of functors $\{F_i\colon \catC \to \catD\}_{i \in I}$ is said to be \emph{jointly conservative} if a morphism $f\colon c\to c'$ in $\catC$ is an isomorphism whenever $F_i(f)$ is so for each $i \in I$.
\end{definition}

If $\{F_i\}_i$ has only one element $F$, then that $\{F_i\}_i=\{F\}$ is jointly conservative just means that $F$ is conservative.

\begin{defprop}\label{defprop:strong_generator}
	For a family of objects $S$ of a category $\catC$, the following conditions are equivalent. We call $S$ a \emph{strongly generating set of objects} if it satisfies one of them (hence all of them).
	\begin{enumerate}
		\item The family of functors $\{\Hom(s,\mplaceholder)\}_{s\in S}$ is jointly faithful and jointly conservative.
		\item The left Kan extension $\Lan_F \yoneda\colon \catC\to \Psh{\catS}\,; c\mapsto \res{\Hom(\mplaceholder,c)}{\catS^\op}$ is faithful and conservative, where $F\colon\catS \hookrightarrow \catC$ denotes the inclusion functor.
	\end{enumerate}

\end{defprop}

\begin{proof}
	Given a morphism $f$ of $\catC$, we have $\Hom(s,f)=f\circ\mplaceholder=\Lan_F(f)(s)$. Hence it holds that $\Hom(s,f)$ is an isomorphism for all $s\in S$ if and only if $\Lan_F(f)$ is so, from which the assertion follows.
\end{proof}

We call an object $G$ a \emph{strong generator} if $S=\{G\}$ is a strongly generating set of objects. Obviously, a strongly generating set of objects is a generating set of objects.

\begin{proposition}\label{cor:strong_generator_in_cat_with_equalizer}
	For a family of objects $S$ of a category $\catC$ with equalizers, the following conditions are equivalent.
	\begin{enumerate}
		\item $S$ is a strongly generating set of objects.
		\item The family of functors $\{\Hom(s,\mplaceholder)\}_{s\in S}$ is jointly conservative.
		\item The left Kan extension $\Lan_F \yoneda\colon \catC \to \Psh{\catS}$ is conservative, where $F\colon\catS \hookrightarrow \catC$ denotes the inclusion functor.
	\end{enumerate}
\end{proposition}

\begin{proof}
	To deduce (ii) from (i), it is sufficient to prove that $F=\prod_{s\in S}\Hom(s,\mplaceholder)\colon\catC \to \Set$ is faithful if it is conservative. For morphisms $f,g\colon c\to d$ in $\catC$, suppose $F(f)=F(g)$. Since $\catC$ admits equalizers, we can take the equalizer $e$ of $f,g$:
	\[
	\begin{tikzcd}
	e \arrow{r} & c \arrow[shift left=0.7ex]{r}{f} \arrow[shift right=0.7ex]{r}[swap]{g} & d. 
	\end{tikzcd}
	\]
	Then, since $F$ preserves limits, the diagram
	\[
	\begin{tikzcd}
	F(e) \arrow{r} & F(c) \arrow[shift left=0.7ex]{r}{F(f)} \arrow[shift right=0.7ex]{r}[swap]{F(g)} & F(d)
	\end{tikzcd}
	\]
	also forms an equalizer. Thus $F(e)\to F(c)$ is an isomorphism by the assumtion that $F(f)=F(g)$. The conservativity of $F$ implies that $e\to c$ is an isomorphism. Therefore we have $f=g$, which shows that $F$ is faithful.
	
	The other implications are straightforward.
\end{proof}


For cocomplete abelian categories, we have various characterizations of a generating set of objects as follows.

\begin{proposition}\label{prop:equivalent_condition_of_generators_in_abelian_cat}
	For a family of objects $S$ of a cocomplete abelian category $\catA$, the following conditions are equivalent.
	\begin{enumerate}
		\item $S$ is a generating set of objects.
		\item $S$ is a strongly generating set of objects.
		\item $\{\Hom_\catA(s,\mplaceholder)\colon \catA\to \Ab\}_{s\in S}$ is jointly conservative.
		\item For any $A\in \catA$, there is an exact sequence
		\[ \bigoplus_{j \in J}s_j \rightarrow \bigoplus_{i \in I} s_i \rightarrow A \rightarrow 0 \]
		such that each $s_i, s_j$ is in $S$.
	\end{enumerate}
\end{proposition}

\begin{proof}
	To see that the condition (i) is equivalent to (ii), it is sufficient to show that the functor $F=\prod_{s\in S}\Hom(s,\mplaceholder)\colon\catA \to \Set$ is conservative if it is faithful. Given a morphism $f$ of $\catA$, suppose $F(f)$ is invertible. In particular, $F(f)$ is monic and epic, and hence $f$ is so because faithful functors reflect monomorphisms and epimorphisms. Hence $f$ is invertible since $\catA$ is abelian.
	
	The equivalence of (ii) and (iii) follows from \cref{cor:strong_generator_in_cat_with_equalizer}.
	
	The implication from (iv) to (i) can be verified by \cref{defprop:generator}~(v).
	
	To deduce (iv) from (i), let us take an object $A \in \catA$. By \cref{defprop:generator}~(v), there is an epimorphism $\gamma_A\colon \bigoplus_i s_i \to A$ with $s_i\in S$. Setting $B\coloneqq \Kernel(\gamma_A)$, we have an exact sequence
	\[
	\begin{tikzcd}
	0 \arrow{r} & B \arrow{r}{f}  & \bigoplus_i s_i \arrow{r}{\gamma_A} & A \arrow{r} & 0.
	\end{tikzcd}
	\]
	By \cref{defprop:generator}~(v) again, we get an epimorphism $\gamma_B\colon \bigoplus_j s_j\to B$ with $s_j \in S$. Then the sequence
	\[
	\begin{tikzcd}
	\bigoplus_j s_j \arrow{r}{f\circ\gamma_B} & \bigoplus_i s_i \arrow{r}{\gamma_A} & A \arrow{r} & 0
	\end{tikzcd}
	\]
	is exact.
\end{proof}

\begin{remark}
	Note that there is no difference between a generator and a generating set of objects in cocomplete abelian categories.
	More precisely, in a cocomplete abelian category $\catA$, the existence of a generator is equivalent to that of a generating set of objects. Indeed, given a generating set $S=\{s_i\}_{i \in I}$ of objects, we can verify that $G \coloneqq \bigoplus_{i \in I} s_i$ is a generator.
\end{remark}

\subsection{Locally finitely presentable bases and finite limits in enriched categories}\label{subsec:flp_base_and_finite_limits}

In this subsection, we review the adaptation of the notion of finite limits from usual categories to enriched categories.
Main references are \cite{Kelly:1982Structures} and \cite{Borcuex-Quinteiro-Rosicky:1998}.

\begin{definition}
	A category $\catJ$ is said to be \emph{filtered} if it fulfolls the following.
	\begin{enumerate}
		\item $\catJ$ is not empty.
		\item For any pair $i,j\in \catJ$ of objects, there are an object $k \in \catJ$ and morphisms $i \to k$, $j \to k$.
		\item For any pair $f,g\colon i \to j$ of morphisms, there is a morphism $h\colon j \to k$ such that $h\circ f=h\circ g$.
	\end{enumerate}
	We call a functor $F\colon \catJ\to\catC$ a \emph{filtered diagram} if $\catJ$ is a small filtered category, and we refer to its colimit as a \emph{filtered colimit}.
\end{definition}

Viewing posets as categories, we observe that filtered posets are just directed posets.

\begin{definition}
	For a category $\catC$ with filtered colimits, an object $C\in \catC$ is said to be \emph{finitely presentable} if the representable functor $\Hom_\catC(C,\mplaceholder)\colon \catC\to \Set$ preserves filtered colimits.
\end{definition}

\begin{example}
	Finitely presentable objects of the category $\Mod(R)$ of modules over a commutative ring $R$ are nothing but finitely presented modules.
\end{example}

\begin{definition}[{\cite[Def.\ 1.1]{Borcuex-Quinteiro-Rosicky:1998}}]
	A symmetric monoidal closed category $\moncatV=(\basecatV,\otimes,I)$ is called a \emph{locally finitely presentable base} if it satisfies the following conditions.
	\begin{enumerate}
		\item $\basecatV$ is locally finitely presentable, which means that it is cocomplete and has a strongly generating set of finitely presentable objects.
		\item The unit object $I$ is finitely presentable.
		\item The tensor product $X\otimes Y$ of finitely presentable objects $X,Y\in \basecatV$ is finitely presentable.
	\end{enumerate}
	We write $\basecatV_\fp$ for the full subcategory of finitely presentable objects.
\end{definition}

The cartesian closed category $\Set$ of sets is an example of locally finitely presentable bases. Later in \cref{subsec:Grothendieck_cosmoi} we will prove that the monoidal categories $\Ab$ and $\Ch$ of abelian groups and of complexes of abelian groups, respectively, also serve as examples.

\begin{definition}
	Let $\moncatV$ be a locally finitely presentable base.
	\begin{enumerate}
		\item A $\moncatV$-category $\catJ$ is said to be \emph{finite} if it satisfies the following conditions.
		\begin{itemize}
			\item $\obj(\catJ)$ is a finite set.
			\item For any pair $j,k\in \catJ$ of objects, the Hom object $\catJ(j,k)\in \basecatV$ is finitely presentable.
		\end{itemize}
		
		\item A $\moncatV$-functor $W\colon \catJ\to\catV$ is said to be \emph{finite} if it satisfies the following conditions.
		\begin{itemize}
			\item $\catJ$ is a finite $\moncatV$-category.
			\item For any object $j\in \catJ$, the image $W(j)\in \basecatV$ is finitely presentable.
		\end{itemize}
		
		\item We define a \emph{finite limit} as a limit weighted by a finite $\moncatV$-functor. We say that a $\moncatV$-category $\catC$ is \emph{finitely complete} if it admits all finite limits.
		
		\item A $\moncatV$-functor $F\colon \catC \to \catD$ is said to be \emph{left exact} (or \emph{finitely continuous}) if the domain $\catC$ admits finite limits and $F$ preserves them.
	\end{enumerate}
\end{definition}

Note that the cotensor product $X\cotensor\mplaceholder$ with a finitely presentable object $X\in \basecatV_\fp$ is a finite limit.

\begin{proposition}[{\cite[Prop.\ 4.3]{Kelly:1982Structures}}]\label{prop:finite_limit_is_finconicallimit_and_finitecotensor}
	Let $\moncatV$ be a locally finitely presentable base. A $\moncatV$-category $\catC$ admits all finite limits if and only if it admits both finite conical limits on ordinary finite categories and cotensor products with finitely presentable objects $X\in \basecatV_\fp$.
\end{proposition}

\begin{proposition}\label{prop:condition_for_left_exact}
	Let $\moncatV$ be a locally finitely presentable base and $F\colon \catC \to \catD$ a $\moncatV$-functor between $\moncatV$-categories. Suppose that $\catC$ admits finite limits. Then $F$ is left exact if and only if it preserves both finite conical limits on ordinary finite categories and cotensor products with finitely presentable objects $X\in \basecatV_\fp$.
\end{proposition}

\begin{proof}
	The claim follows from \cref{prop:finite_limit_is_finconicallimit_and_finitecotensor}.
\end{proof}

\subsection{Dualizable objects}\label{subsec:dualizable_objects}

In this subsection, we recall an important notion of finiteness for objects of monoidal categories.

\begin{definition}
	Let $\moncatV=(\basecatV,\otimes,I)$ be a monoidal category. A pair $(X,Y)$ of objects is called a \emph{dual pair} if there are morphisms
	\[ \eta\colon I \to Y \otimes X, \qquad \varepsilon\colon X\otimes Y\to I \]
	in $\basecatV$ such that the following diagrams commute:
	\[
	\begin{tikzcd}
	X \arrow{r}{\cong} \arrow{rrdd}[swap]{\id_X} & X\otimes I \arrow{r}{X\otimes\eta} & X\otimes Y\otimes X \arrow{d}{\varepsilon\otimes X}\\
	& & I\otimes X \arrow{d}{\cong} \\
	& & X,
	\end{tikzcd}
	\qquad
	\begin{tikzcd}
	Y \arrow{r}{\cong} \arrow{rrdd}[swap]{\id_Y} & I\otimes Y \arrow{r}{\eta\otimes Y} & Y\otimes X\otimes Y \arrow{d}{Y\otimes \varepsilon}\\
	& & Y\otimes I \arrow{d}{\cong} \\
	& & Y.
	\end{tikzcd}
	\]
	If $(X,Y)$ is a dual pair, then we call $X$ a \emph{left dual} to $Y$ and $Y$ a \emph{right dual} to $X$.
\end{definition}

\begin{proposition}[{\cite[III\S1, Thm.\ 1.6]{Lewis-May-Steinberger:1986}, \cite[Prop.\ 2.10.8]{Etingof-Gelaki-Nikshych-Ostrik:2015}}]
	Let $\moncatV$ be a monoidal category and $(X,Y)$ a dual pair. Then for any $Z,W\in \basecatV$, we have natural bijections
	\begin{equation*}
	\begin{split}
	\Hom_V(Z,W\otimes X) &\cong \Hom_V(Z\otimes Y,W), \\
	\Hom_V(X\otimes Z,W) &\cong \Hom_V(Z,Y\otimes W);
	\end{split}\tag{$\Diamond$}
	\end{equation*}
	These bijections lead to the adjunctions $\mplaceholder\otimes Y\dashv\mplaceholder\otimes X$ and $X\otimes\mplaceholder\dashv Y\otimes\mplaceholder$.
	Conversely, if there is either of the natural bijections in ($\Diamond$) for $X,Y\in \moncatV$, then $(X,Y)$ becomes a dual pair.
\end{proposition}

\begin{remark}\label{rem:dual_pair_in_cosmos}
	If $(X,Y)$ is a dual pair in a symmetric monoidal closed category $\moncatV$, then so is the pair $(Y,X)$, and we have isomorphisms of functors
	\[ [Y,\mplaceholder] \cong X\otimes \mplaceholder,\qquad [X,\mplaceholder]\cong Y\otimes\mplaceholder. \]
	In particular, it holds that $Y\cong [X,I]$ and $[X,I]\otimes\mplaceholder\cong [X,\mplaceholder]$.
\end{remark}

\begin{definition}
	An object $X$ of a symmetric monoidal closed category $\moncatV$ is said to be \emph{dualizable} if there is an object $X^*\in \basecatV$ such that $(X,X^*)$ is a dual pair. \cref{rem:dual_pair_in_cosmos} shows that $X^*\cong [X,I]$ if it exists. We write $\basecatV_\du$ for the full subcategory of dualizable objects.
\end{definition}

Note that if there is a natural isomorphism $[X,I]\otimes\mplaceholder\cong [X,\mplaceholder]$, then the pair $(X,[X,I])$ is a dual pair, and hence $X$ is dualizable (\cite{Lewis-May-Steinberger:1986}).

\begin{example}
	Dualizable objects of the symmetric monoidal closed category $\Mod(R)$ of modules over a commutative ring $R$ are nothing but finitely generated projective modules.
\end{example}

\begin{proposition}\label{prop:tensoring_with_dualizable_is_absolute}
	Let $\moncatV$ be a symmetric monoidal closed category. For $X\in \basecatV$, the following conditions are equivalent.
	\begin{enumerate}
		\item $X$ is dualizable.
		\item For any $\moncatV$-category $\catC$, the cotensor product $X\cotensor C$ with an object $C\in \catC$ is an \emph{absolute limit}, that is, preserved by arbitrary $\moncatV$-functors out of $\catC$.
		\item For any $\moncatV$-category $\catC$, the tensor product $X\tensor C$ with an object $C\in \catC$ is an \emph{absolute colimit}, that is, preserved by arbitrary $\moncatV$-functors out of $\catC$.
	\end{enumerate}
\end{proposition}

\begin{proof}
	This assertion is a special case of \cite{Street:1983}.
	
	Here we only verify the implication (i) $\Rightarrow$ (ii). Consider a $\moncatV$-functor $S\colon \catC\to \catD$. For $D\in \catD$, we have
	\begin{alignat*}{2}
	\catD(D,S(X\cotensor C)) &\cong \int^{C'\in \catC} \catC(C',X \cotensor C) \otimes \catD(D,SC') & \qquad & \text{by \cref{prop:coYoneda_lemma},} \\
	&\cong \int^{C'\in \catC} \catV(X,\catC(C',C)) \otimes \catD(D,SC') & & \\
	&\cong \int^{C'\in \catC} [X,I] \otimes \catC(C',C) \tensor \catD(D,SC') & & \text{by \cref{rem:dual_pair_in_cosmos},} \\
	&\cong [X,I] \otimes \int^{C'\in \catC} \catC(C',C) \tensor \catD(D,SC') & & \\
	&\cong [X,I] \otimes \catD(D,SC) & & \text{by \cref{prop:coYoneda_lemma},} \\
	&\cong \catV(X,\catD(D,SC)) & & \text{by \cref{rem:dual_pair_in_cosmos},} \\
	&\cong \catD(D,X\cotensor SC). & &
	\end{alignat*}
	These isomorphisms are natural in $D$, and hence we have $S(X\cotensor C) \cong X\cotensor SC$ by the Yoneda lemma.
\end{proof}

\section{Grothendieck enriched categories}\label{sec:Grothendieck_enriched_categories}

\subsection{The Gabriel-Popescu theorem}\label{subsec:the_Gabriel-Popescu_theorem}

We first recall the definition of Grothendieck categries.

\begin{definition}
	A \emph{Grothendieck category} $\catA$ is an abelian category that possesses the following properties:
	\begin{enumerate}
		\item it admits all small coproducts;
		\item filtered colimits are exact;
		\item it has a generator $G$, which means that the functor $\catA(G,\mplaceholder)\colon \catA \to \Ab$ is faithful.
	\end{enumerate}
\end{definition}

Since abelian categories admit coequalizers, Grothendieck categories are cocomplete by the condition (i). Remember that an abelian category has the unique $\Ab$-enrichment.

\begin{example}
	The category $\Mod(R)$ of modules over a ring $R$ is a Grothendieck category.
\end{example}

\begin{example}
	The category $\Qcoh(X)$ of quasi-coherent sheaves on a scheme $X$ is a Grothendieck category (see \cite[\href{https://stacks.math.columbia.edu/tag/077K}{Tag 077K}]{stacks-project}).
\end{example}

Before stating a slightly generalized version of the Gabriel-Popescu theorem (\cite{Gabriel-Popescu:1964}), we remark that a cocomplete abelian category $\catA$ is also cocomplete as an $\Ab$-enriched category (see \cite[Prop.\ 3.76]{Kelly:1982Basic}), so that an additive functor $F\colon \catC\to \catA$ yields the nerve-and-realization adjunction $\Lan_\yoneda F \dashv \Lan_F\yoneda \colon \Mod(\catC) \to \catA$ by \cref{thm:ubiquitus_adjunction}.

\begin{theorem}[{The Gabriel-Popescu theorem~\cite{Gabriel-Popescu:1964,Kuhn:1994}}]\label{thm:gabriel_popescu}
	Let $\catA$ be a Grothendieck category that has a generating set $\catC$ of objects. We regard $\catC$ as a full additive subcategory and write $F\colon \catC\hookrightarrow\catA$ for the inclusion additive functor. Then the following assertions on the nerve-and-realization adjunction associated with $F$ hold.
	\begin{enumerate}
		\item The right adjoint $\Lan_F\yoneda$ is fully faithful.
		\item The left adjoint $\Lan_\yoneda F$ is left exact.
	\end{enumerate}
\end{theorem}

This form of the Gabriel-Popescu theorem is given in \cite[Thm.\ 2.1]{Kuhn:1994}; see also \cite{Lowen:2004}.
The Gabriel-Popescu theorem is an extrinsic characterization of Grothendieck categories.

\subsection{Grothendieck cosmoi}\label{subsec:Grothendieck_cosmoi}

Sometimes a complete and cocomplete symmetric monoidal closed category $\moncatV=(\basecatV,\otimes,I)$ is called a \emph{cosmos}. If moreover $\moncatV$ is a Grothendieck category, then we will call it a \emph{Grothendieck cosmos}. For a Grothendieck cosmos $\basecatV$, let us consider a kind of finiteness conditions as follows:
\begin{enumerate}
	\item[] (C1) the unit object $I\in \basecatV$ is finitely presentable;
	\item[] (C2) $\basecatV$ has a generating set $\{g_j\}_{j \in J}$ of dualizable objects.
\end{enumerate}
Note that $\{g_j\}_{j \in J}$ is also strongly generating by \cref{prop:equivalent_condition_of_generators_in_abelian_cat}.

\begin{lemma}\label{lem:construction_of_dualizables}
	Let $\moncatV$ be a Grothendieck cosmos. If $X,Y\in \basecatV$ are dualizable, then so are $X\oplus Y$, $X\otimes Y$, and $[X,Y]$.
\end{lemma}

\begin{proof}
	See \cite[Lem.\ 6.7]{Holm-Odabasi:2019}.
\end{proof}

\begin{proposition}[{\cite[Prop.\ 6.9 (a)]{Holm-Odabasi:2019}}]\label{prop:C1_C2_implies_lfp_base}
	If a Grothendieck cosmos $\moncatV$ satisfies the conditions (C1) and (C2), then it is a locally finitely presentable base.
\end{proposition}

\begin{proof}
	We begin by proving that dualizable objects of $\moncatV$ are finitely presentable. Take a dualizable object $X\in \basecatV$. Then we have composites of isomorphisms
	\[ \Hom_\basecatV(X,\mplaceholder) \cong \Hom_\basecatV(I\otimes X,\mplaceholder) \cong \Hom_\basecatV(I,[X,\mplaceholder]) \cong \Hom_\basecatV(I,[X,I]\otimes\mplaceholder). \]
	Since $I$ is finitely presentable, $\Hom_\basecatV(I,\mplaceholder)$ preserves filtered colimits. Hence $\Hom_\basecatV(X,\mplaceholder)$ also does, which shows that $X$ is finitely presentable.
	From this, we can conclude that $\basecatV$ has a (strongly) generating set of finitely presentable objects and hence is locally finitely presentable.
	
	Second, we claim that $X\in \basecatV$ is finitely presentable if and only if there is an exact sequence
	\begin{equation*}
	P_1 \to P_0 \to X \to 0 \tag{$\ast$}
	\end{equation*}
	with $P_1,P_0$ dualizable. Indeed, if $X$ is finitely presentable, it follows from \cite[Prop.\ V.3.4]{Stenstrom:1975} that there is an exact sequence in \cref{prop:equivalent_condition_of_generators_in_abelian_cat}~(iv) where the indexing sets $I,J$ are finite. Since by \cref{lem:construction_of_dualizables} finite direct products of dualizable objects are dualizable again, we have a desired sequence. On the other hand, if $X$ has a presentation like ($\ast$), dualizable objects $P_0,P_1$, and hence $X$, are finitely presentable.
	
	It remains to check that monoidal products of finitely presentable objects are finitely presentable. If $X,Y\in \basecatV$ are finitely presentable, there exist two exact sequences
	\[ P_1 \to P_0 \to X \to 0, \quad Q_1 \to Q_0 \to Y \to 0 \]
	with $P_1,P_0,Q_1,Q_0$ dualizable. Then it follows from \cite[ChapterII \S3.6 Prop.\ 6]{Bourbaki:1989Algebra1} that
	\[ (P_1\otimes Q_0)\oplus(P_0\otimes Q_1) \to P_0\otimes Q_0 \to X \otimes Y\to 0 \]
	is exact. Thus $X\otimes Y$ is finitely presentable, because both $(P_1\otimes Q_0)\oplus(P_0\otimes Q_1)$ and $P_0\otimes Q_0$ are dualizable.
\end{proof}

\begin{example}
	The category $\Mod(R)$ of modules over a commutative ring $R$ is a Grothendieck category with $R$ a generator and becomes a symmetric monoidal closed category with the ordinary tensor product of modules. This Grothendieck cosmos $\Mod(R)$ satisfies the conditions (C1) and (C2).
\end{example}

\begin{example}
	For a ringed space $(X,\mO_X)$, let $\Mod(X)$ denote the category of sheaves of $\mO_X$-modules on $X$. For an open subset $U\subseteq X$, let a sheaf $S_U$ be the extension by zero of $\mO_U=\res{\mO_X}{U}$. Then $\Mod(X)$ is a Grothendieck category with $\{S_U\mid U\subset X \text{ is open}\}$ a generating set of objects (\cite{Lowen:2004}) and becomes a symmetric monoidal closed category with the tensor product $\otimes_X$ of $\mO_X$-modules and Hom sheaf $\intHom_X$.
	
	If $X$ has a base of compact open subsets, then $\Mod(X)$ is a locally finitely presentable category with $\{S_U\mid \text{$U$ is compact and open}\}$ a generating set of finitely presentable objects. However, unless $X$ itself is compact, $\mO_X$ is not finitely presentable and hence $\Mod(X)$ does not satisfy the condition (C1) (\cite{Prest-Ralph:2010}).
\end{example}

\begin{example}\label{example:G-cosmos_of_quasi-coh}
	The category $\Qcoh(X)$ of quasi-coherent sheaves on a scheme $X$ is a Grothendieck category (\cite[\href{https://stacks.math.columbia.edu/tag/077K}{Tag 077K}]{stacks-project}). The inclusion functor $\Qcoh(X)\to\Mod(X)$ has a right adjoint $Q_X$, called the \emph{coherator} (\cite[\href{https://stacks.math.columbia.edu/tag/08D6}{Tag 08D6}]{stacks-project}). Then the tensor product of $\mO_X$-modules together with $\intHom_X^\qc\coloneqq Q_X \intHom_X$ makes $\Qcoh(X)$ into a symmetric monoidal closed category.
	
	If $X$ is quasi-compact and quasi-separated, then $\Qcoh(X)$ satisfies the condition (C1). Furthermore $\Qcoh(X)$ also fulfills the condition (C2) if $X$ is projective (see \cite[Example 6.3]{Holm-Odabasi:2019}).
\end{example}

\begin{example}\label{example:Ch(R)_satisfy_C1andC2}
	For a commutative ring $R$, let $\Ch(R)$ be the category of cochain complexes of $R$-modules. Define the \emph{sphere complex} $S(R)$ and the \emph{disk complex} $D(R)$ to be complexes concentrated in degree $0$ and in degrees $-1$ and $0$, respectively, as follows:
	\begin{alignat*}{2}
	S(R)\colon \qquad & \cdots\to 0 \to 0 \to R \to 0 \to \cdots, & \qquad & \\
	D(R)\colon \qquad & \cdots \to 0 \to R \xrightarrow{\id} R \to 0 \to \cdots. & &
	\end{alignat*}
	Then $\Ch(R)$ is a Grothendieck category with $\{D(R)[n] \mid n \in \Z\}$ a generating set of objects. It also forms a symmetric monoidal closed category $(\Ch(R),\tten_R,S(R),\tHom_R)$, where $\tten_R$ is the total tensor product of complexes.
	
	It is not hard to see that 
	\[ \Ch(R)_\fp = \{ X\in\Ch(R) \mid \text{$X$ is bounded and each term $X^n$ is finitely presentable} \} \]
	and \cite[Prop.\ 1.6]{Dold-Pippe:1984} shows that
	\[ \Ch(R)_\du = \{ X\in\Ch(R) \mid \text{$X$ is bounded and each term $X^n$ is finitely generated and projective} \}. \]
	In particular, $S(R)$ and $D(R)$ are finitely presentable and dualizable.
	Therefore the Grothendieck cosmos $(\Ch(R),\tten_R)$ satisfies the conditions (C1) and (C2).
\end{example}

\begin{example}\label{example:Ch(V)_satisfy_C1andC2}
	More generally, for any Grothendieck cosmos $\moncatV$, the category $\Ch(\moncatV)$ of cochain complexes in $\moncatV$ becomes a Grothendieck cosmos with the total tensor product and the total Hom complex whose unit object is the sphere complex $S(I)$ of the unit object $I$ of $\moncatV$ (\cite[Thm.\ 3.2]{Garkusha-Jones:2018}). If $\moncatV$ has a generating set $\{g_j\}_{j \in J}$ of objects, then $\{D(g_j)[n] \mid j\in J, n \in \Z\}$ forms a generating set of objects in $\Ch(\moncatV)$.
	
	If $\moncatV$ satisfies the conditions (C1) and (C2), then $\Ch(\moncatV)$ also does. In fact, $S(X)\in \Ch(\moncatV)$ is finitely presentable if $X\in \moncatV$ is so, and $D(X)[n]\in \Ch(\moncatV)$ is dualizable if $X\in \moncatV$ is so.
\end{example}

\begin{example}
	For a commutatve ring $R$, we can define another tensor product on $\Ch(R)$ (\cite{Enochs-Rozas:1997}). For two complexes $M$ and $N$, the \emph{modified tensor product} $M\mten_R N$ is defined as the complex
	\[ \frac{(M\tten_R N)^n}{B^n(M\tten_R N)} \to \frac{(M\tten_R N)^{n+1}}{B^{n+1}(M\tten_R N)}, \quad \overline{x\otimes y}\mapsto \overline{d_M(x)\otimes y}. \]
	The \emph{modified Hom complex} $\mHom_R(M,N)$ is the complex such that
	\[ \Hom_{\Ch(R)}(M,N[n]) \to \Hom_{\Ch(R)}(M,N[n+1]), \quad f \mapsto 
	((-1)^n d_N \circ f^m)_{m \in \Z}. \]
	Then $\Ch(R)$ becomes a symmetric monoidal closed category $(\Ch(R),\mten_R,D(R),\mHom_R)$.
	Each $D(R)[n]$ is dualizable also in this monoidal category, and the Grothendieck cosmos $(\Ch(R),\mten_R)$ satisfies the conditions (C1) and (C2).
\end{example}

\subsection{Grothendieck $\moncatV$-enriched categories}\label{subsec:Grothendieck_enriched_categories}

Let $\moncatV$ be a cosmos, i.e., a complete and cocomplete symmetric monoidal closed category.

\begin{definition}\label{def:V-generators}
	Let $\catA$ be a $\moncatV$-category. We say that a full subcategory $\catC$ is a \emph{$\moncatV$-generating set of objects} if it is small and the underlying functor $(\Lan_F\yoneda)_0$ of the left Kan extension of the Yoneda embedding $\yoneda\colon \catC\to[\catC^\op,\catV]$ along the inclusion $\moncatV$-functor $F\colon\catC\hookrightarrow\catA$ is faithful.
\end{definition}

\begin{proposition}\label{prop:image_of_generating_functor_is_generating_set}
	Let $F\colon \catC\to\catA$ be a $\moncatV$-functor and suppose $\catC$ is small. Write the full subcategory of images of $F$ as $\catA'=\{Fc\mid c\in \catC\}$. If $(\Lan_F\yoneda_{\catC})_0\colon \catA_0\to[\catC^\op,\catV]_0$ is faithful, then $\catA'$ is a $\moncatV$-generating set of objects in $\catA$.
\end{proposition}

\begin{proof}
	Since $\catC$ is small, $\catA'$ is also small. If $\iota\colon\catA'\hookrightarrow\catA$ denotes the inclusion functor, then we see
	\[
	\Lan_F \yoneda_{\catC} \cong F^* \circ \Lan_\iota \yoneda_{\catA'} : \catA \to [{\catA'}^\op,\catV] \to [\catC^\op,\catV].
	\]
	By assumption, $(\Lan_F\yoneda_\catC)_0$, and hence $(\Lan_\iota\yoneda_{\catA'})_0$, is faithful. This shows that $\catA'$ is a $\moncatV$-generating set of objects in $\catA$.
\end{proof}

\begin{proposition}\label{prop:generators_form_V-generators}
	Let $\catA$ be a  $\moncatV$-category and $\catC$ a full subcategory. If $\catC_0$ is a generating set of objects in $\catA_0$, then $\catC$ is a $\moncatV$-generating set of objects in $\catA$.
\end{proposition}

\begin{proof}
	Let $F\colon\catC\hookrightarrow\catA$ be the inclusion functor. By \cref{prop:Lan_of_yoneda}, the underlying functor $(\Lan_F\yoneda)_0\colon \catA_0 \to \Fun(\catC^\op,\catV)$ of the left Kan extension is given by the following correspondances:
	\begin{align*}
	\catA_0 \ni A &\mapsto 
	(\Lan_F\yoneda)_0(A)=\res{\catA(\mplaceholder,A)}{\catC}, \\
	\catA_0(A,B) \ni f &\mapsto (\Lan_F\yoneda)_0(f)=f\circ\mplaceholder\colon \res{\catA(\mplaceholder,A)}{\catC}\to\res{\catA(\mplaceholder,B)}{\catC}.
	\end{align*}
	For $c\in \catC_0$, we have $\catA_0(c,f)= U(\catA(c,f))=U(f\circ\mplaceholder)$. In order to prove $(\Lan_F\yoneda)_0$ is faithful, we assume $(\Lan_F\yoneda)_0(f)=(\Lan_F\yoneda)_0(g)$ for any pair $f, g\in \catA_0(A,B)$. Then we have $\catA_0(c,f)=\catA_0(c,g)$ for all $c\in \catC_0$, which yields $f=g$ since $\catC_0$ is a generating set of objects of $\catA_0$. Thus $\catC$ is a $\moncatV$-generating set of objects.
\end{proof}

Inspired by the Gabriel-Popescu theorem, we define an enriched version of Grothendieck categories as follows.

\begin{definition}\label{def:Grothendieck_V-category}
	Let $\moncatV$ be a locally finitely presentable base. A $\moncatV$-category $\catA$ is said to be a \emph{Grothendieck $\moncatV$-category} if there exist a small $\moncatV$-category $\catC$ and a $\moncatV$-adjunction
	\[
	\begin{tikzpicture}
	\node at (-0.2,0) {$S:[\catC^\op,\catV]$};
	\node at (3.1,0) {$\catA : T$};
	\draw[->] (0.9,0.16) to (2.5,0.16);
	\draw[->] (2.5,-0.16) to (0.9,-0.16);
	\node at (1.7,0) {$\perp$};
	\end{tikzpicture}
	\]
	such that
	\begin{enumerate}
		\item the right adjoint $T$ is fully faithful and 
		\item the left adjoint $S$ is left exact.
	\end{enumerate}
\end{definition}

Remember that Grothendieck cosmoi with the conditions (C1) and (C2) are locally finitely presentable bases by \cref{prop:C1_C2_implies_lfp_base}.

\begin{remark}\label{remark:Ivan's_comments}
\cref{def:Grothendieck_V-category} also appears in \cite{Garner-Lack:2012} under the name of \emph{$\moncatV$-topoi}.
In \cite{Garner-Lack:2012} full subcategories of a category whose inclusion functor has a left exact left adjoint, as in \cref{def:Grothendieck_V-category}, are called \emph{localizations}.
The study of localizations of a category has a long history; see the answers~\cite{DiLiberti:MathOverflow349948,DiLiberti:MathOverflow389388} to questions on mathoverflow for a list of references.
\end{remark}

\begin{remark}
If we take $\moncatV=\Ab$, then Grothendieck $\Ab$-categories are the ordinary Grothendieck categories and if $\moncatV=\Set$, then Grothendieck $\Set$-categories are the Grothendieck topoi, the categories of sheaves on sites. This indicates that Grothendieck categories can be seen as the categories of linear sheaves on linear sites.
In fact, Gabriel topologies are precisely the $\Z$-linear version of Grothendieck topologies.
Lowen \cite{Lowen:2016} and Ramos Gonz\'alez \cite{RamosGonzalez:2018} study Grothendieck categories from the viewpoint of the theory of linear sheaves. See also \cite[\S2.1]{DiLiberti-RamosGonzalez:2021}.
\end{remark}

For a locally finitely presentable base $\moncatV$, the $\moncatV$-category $\catV$ itself and the presheaf category $[\catC^\op,\catV]$ on a small $\moncatV$-category $\catC$ are examples of Grothendieck $\moncatV$-categories.
Note that since Grothendieck $\moncatV$-categories are reflective full subcategories of presheaf categories, they are complete and cocomplete by \cite[Prop.\ 3.75]{Kelly:1982Basic}.

\cref{def:Grothendieck_V-category} of Grothendieck $\moncatV$-categories is extrinsic. On the other hand, under the assumption that $\moncatV$ is a nice Grothendieck cosmos, the following theorem gives an intrinsic characterization.

\begin{theorem}\label{thm:main_theorem}
	Let $\moncatV$ be a Grothendieck cosmos with the conditions (C1) and (C2). Then a $\moncatV$-category $\catA$ is a Grothendieck $\moncatV$-category if and only if it fulfills the following conditions.
	\begin{enumerate}
		\item $\catA$ is cocomplete.
		\item $\catA$ is finitely complete.
		\item $\catA$ has a $\moncatV$-generating set $\catC$ of objects.
		\item The homomorphism theorem holds in $\catA$, that is, for any morphism $f$ in $\catA_0$ the canonical map $\Coker(\Kernel(f))\to\Kernel(\Coker(f))$ is an isomorphism.
		\item Conical filtered colimits are left exact, which means that for any ordinary filtered category $\catJ$, the colimit $\moncatV$-functor $\colim_\catJ \colon [\catJ_\moncatV,\catA]\to\catA$ preserves finite limits.
	\end{enumerate}
\end{theorem}

Note that $\catA_0$ has the natural $\Ab$-enriched structure. Hence we can make sense of the (co)kernel $\Kernel$ (resp.\ $\Coker$) of a morphism as the (co)equalizer with the zero morphism, whose existence is guaranteed by the finite (co)completeness of $\catA$.

\begin{proposition}[{\cite[Thm.\ 4.2]{AlHwaeer-Garkusha:2016}}]\label{prop:Grothen_cat_of_enriched_functors}
	Let $\moncatV$ be a Grothendieck cosmos with $\{g_j\}_{j \in J}$ a generating set of objects. For a small $\moncatV$-category $\catC$, the category $\Fun(\catC,\catV)$ of $\moncatV$-functors is a Grothendieck category with $\{g_j\tensor \catC(c,\mplaceholder)\mid j \in J,\, c\in \catC\}$ a generating set of objects.
\end{proposition}

\begin{lemma}\label{lem:S_is_left_exact_when_S0_is_so}
	Let $\moncatV$ be a Grothendieck cosmos with the conditions (C1) and (C2). For a finitely complete $\moncatV$-category $\catC$, a $\moncatV$-functor $S\colon \catC\to\catD$ is left exact if and only if its underlying functor $S_0\colon \catC_0\to\catD_0$ is left exact.
\end{lemma}

\begin{proof}
	It is obvious that if $S$ is left exact, then so is $S_0$.
	
	If $S_0$ is left exact, then $S$ preserves conical finite limits on ordinary finite categories. Hence, it is sufficient by \cref{prop:condition_for_left_exact} to show that $S$ preserves cotensor products $X\cotensor c$ with $X\in \basecatV_\fp$ and $c\in \catC$. From the proof of \cref{prop:C1_C2_implies_lfp_base}, for any $X\in \basecatV_\fp$ there is an exact sequence
	\[ P_1 \to P_0 \to X \to 0 \]
	with $P_1,P_0$ dualizable in $\basecatV$. Since the cotensor product forms the $\moncatV$-adjunction $\catC(\mplaceholder,c) \dashv \mplaceholder\cotensor c$, the right adjoint $\mplaceholder\cotensor c\colon \catV^\op\to\catC$ preserves limits, and so we have the exact sequence
	\[ 0\to X\cotensor c \to P_0 \cotensor c \to P_1\cotensor c. \]
	From the left exactness of $S_0$, we also obtain the exact sequence
	\[ 0\to S(X\cotensor c) \to S(P_0\cotensor c) \to S(P_1\cotensor c). \]
	Now, by \cref{prop:tensoring_with_dualizable_is_absolute}, the cotensor products $P_0 \cotensor c, P_1\cotensor c$ are preserved by $S$ because $P_0,P_1$ are dualizable. Thus
	\[ 0\to S(X\cotensor c) \to P_0\cotensor S(c) \to P_1\cotensor S(c) \]
	is also exact, yielding $S(X\cotensor c)\cong X\cotensor S(c)$. This proves that $S$ is left exact.
\end{proof}

Now we are ready to give the proof of \cref{thm:main_theorem}.

\begin{proof}[Proof of \cref{thm:main_theorem}]
	In order to prove 
	the necessity, let $\catA$ be a Grothendieck $\moncatV$-category and $S \dashv T \colon [\catC^\op,\catV]\to \catA$ the $\moncatV$-adjunction of \cref{def:Grothendieck_V-category}. Note that $\catA$ is complete and cocomplete since it is a reflective full subcategory of the presheaf category $[\catC^\op,\catV]$. \cref{thm:adjunction_from_functor_cat} leads to $T\cong \Lan_F \yoneda$, where $F\coloneqq S\circ \yoneda\colon \catC\to \catA$. In particular, $T_0\cong (\Lan_F\yoneda)_0$ is faithful. Thus, using \cref{prop:image_of_generating_functor_is_generating_set}, we see that $\catA$ has $\catA'=\{Fc\mid c\in \catC\}$ as a $\moncatV$-generating set of objects.
	
	The underlying category $\catA_0$ is a Giraud subcategory of the Grothendieck category $\Fun(\catC^\op,\catV)$, so it follows from \cite[Prop.\ X.1.3]{Stenstrom:1975} that $\catA_0$ is also a Grothendieck category. Note that the homomorphism theorem holds in $\catA$ and filtered colimits are left exact. \cref{lem:S_is_left_exact_when_S0_is_so} shows that conical filtered colimits in $\catA$ are left exact as well.
	
	It remains to prove 
	the sufficiency. Let $F\colon \catC\hookrightarrow\catA$ denote the inclusion functor of the generating set $\catC$ of objects. The cocompleteness of $\catA$ leads by \cref{thm:ubiquitus_adjunction} to the $\moncatV$-adjunction, the nerve-and-realization adjunction associated with $F$, as in the following diagram
	\[
	\begin{tikzpicture}[auto]
	\node (hatC) at (0,0) {$[\catC^\op,\catV]$};
	\node (C) at (0,-2) {$\catC$}; \node (D) at (2.8,-2) {$\catA$.};
	
	\draw[->] (C) to node {$\scriptstyle \yoneda$} (hatC);
	\draw[>->] (C) to node[swap] {$\scriptstyle F$} (D);
	\draw[->] (0.9,0) to node {$\scriptstyle \Lan_\yoneda F$}  (2.8,-1.4);
	\draw[->] (2.5,-1.7) to node {$\scriptstyle \Lan_F \yoneda$} node[sloped,pos=0.45] {$\perp$} (0.6,-0.3);
	\end{tikzpicture}
	\]
	If we write $S=\Lan_\yoneda F$ and $T=\Lan_F \yoneda$, then we want to show that $S$ is left exact and that $T$ is fully faithful. For this purpose, it is sufficient by \cref{lem:S_is_left_exact_when_S0_is_so} and \cref{prop:right_adjoint_is_fullyfaithful_iff_its_underlying_is_so} to observe that $S_0$ left exact and $T_0$ is fully faithful. Note that $T_0$ is faithful since $\catC$ is a generating set of objects.
	
	We recall from \cref{prop:Grothen_cat_of_enriched_functors} that $\Fun(\catC^\op,\catV)$ is a Grothendieck category that has $\{g_j \tensor \catC(\mplaceholder,c)\mid j \in J,\, c\in \catC\}$ as a generating set of objects. Let $\catG$ be the full $\catV$-subcategory of $[\catC^\op,\catV]$ whose objects are $\{g_j \tensor \catC(\mplaceholder,c)\mid j \in J,\, c\in \catC\}$, and $G\colon \catG\hookrightarrow [\catC^\op,\catV]$ be its inclusion $\moncatV$-functor. Then the underlying functor $G_0\colon \catG_0 \hookrightarrow \Fun(\catC^\op,\catV)$ becomes an $\Ab$-functor, with which we get the associated nerve-and-realization $\Ab$-adjunction $S'\dashv T'$. The Gabriel-Popescu \cref{thm:gabriel_popescu} states that the left adjoint $S'$ is left exact and the right adjoint $T'$ is fully faithful:
	\[
	\begin{tikzpicture}[auto]
	\node (hatC) at (0,0) {$\Mod(\catG_0)$};
	\node (C) at (0,-2) {$\catG_0$}; \node (D) at (2.9,-2) {$\Fun(\catC^\op,\catV)$.};
	
	\draw[>->] (C) to node {$\scriptstyle \yoneda$} (hatC);
	\draw[>->] (C) to node[swap] {$\scriptstyle G_0$} (D);
	\draw[->] (0.9,0) to node {$\scriptstyle S'$}  (2.8,-1.4);
	\draw[>->] (2.5,-1.7) to node {$\scriptstyle T'$} node[sloped,pos=0.45] {$\perp$} (0.6,-0.3);
	\end{tikzpicture}
	\]
	Now the composite $S\circ G\colon \catG \to [\catC^\op,\catV]\to \catA$ of $\moncatV$-functors is fully faithful. Indeed, for each object $g_j \tensor c$ of $\catG$ we have
	\[ S\circ G(g_j \tensor \catC(\mplaceholder,c))\cong g_j \tensor S(\catC(\mplaceholder,c)) \cong g_j \tensor c. \]
	Thus 
	\begin{alignat*}{2}
	&\catG(g_j \tensor \catC(\mplaceholder,c), g_{j'} \tensor \catC(\mplaceholder,c'))  & &  \\
	&= [\catC^\op,\catV](g_j \tensor \catC(\mplaceholder,c), g_{j'} \tensor \catC(\mplaceholder,c')) & & \\
	&\cong \catV(g_j, [\catC^\op,\catV](\catC(\mplaceholder,c), g_{j'} \tensor \catC(\mplaceholder,c'))) & &\text{by the definition of $g_j\tensor \catC(\mplaceholder,c)$,}\\
	&\cong \catV(g_j, g_{j'} \tensor \catC(c,c')) && \text{by the Yoneda lemma,} \\
	&= \catV(g_j, g_{j'} \tensor \catA(c,c')) &&\\
	&\cong \catV(g_j,\catA(c, g_{j'}\tensor c')) && \text{by \cref{prop:tensoring_with_dualizable_is_absolute},} \\
	&\cong \catA(g_j \tensor c, g_{j'}\tensor c') && \text{by the definition of $g_j\tensor c$,} \\
	&\cong \catA(S\circ G(g_j \tensor \catC(\mplaceholder,c)), S\circ G(g_{j'} \tensor \catC(\mplaceholder,c'))).&&
	\end{alignat*}
	Hence the underlying functor $H=S_0 \circ G_0$ is also fully faithful, so $\catG_0$ can be regarded as a full subcategory $\{g_j \tensor c \mid j \in J,\, c\in \catC\}$ of $\catA_0$ via $H$. Under this identification, $\catG_0$ is a generating set of objects of $\catA_0$, because in the diagram
	\[
	\begin{tikzpicture}[auto]
	\node (ModG) at (-2.8,2) {$\Mod(\catG_0)$};
	\node (hatC) at (0,0) {$\Fun(\catC^\op,\catV)$};
	\node (C) at (-2.8,-2) {$\catG_0$}; \node (D) at (2.8,-2) {$\catA_0$};
	
	\draw[>->] (C) to node {$\scriptstyle \yoneda$} (ModG);
	\draw[>->] (C) to node[swap] {$\scriptstyle H$} (D);
	\draw[->] (0.9,0) to node {$\scriptstyle S_0$}  (2.8,-1.4);
	\draw[->] (2.5,-1.7) to node {$\scriptstyle T_0$} node[sloped,pos=0.45] {$\perp$} (0.6,-0.3);
	
	\draw[->] (-1.9,2) to node {$\scriptstyle S'$}  (0,0.6);
	\draw[>->] (-0.3,0.3) to node {$\scriptstyle T'$} node[sloped,pos=0.45] {$\perp$} (-2.2,1.7);
	\end{tikzpicture}
	\]
	the composite $S_0 \circ S' \circ \yoneda$ is isomorphic to $H\colon \catG_0\to\catA_0$ by \cref{prop:Kan_extension_along_full_faithful} and hence we have $\Lan_H \yoneda \cong T'\circ T_0$, from which we see $\Lan_H \yoneda$ is faithful. 
	By assumption it follows that $\catA_0$ becomes an abelian category with Grothendieck's condition (AB5). Thus it is a Grothendieck category with $\catG_0$ a generating set of objects. Using the Gabriel-Popescu theorem again, we conclude that $S_0\circ S'\cong\Lan_\yoneda H$ is left exact and that $T'\circ T_0 \cong \Lan_H \yoneda$ is fully faithful.
	
	Now we are able to observe that $S_0$ is left exact and $T_0$ is fully faithful. Since $T'$ and $T'\circ T_0$ are fully faithful, so is $T_0$. Next, consider a finite limit $\lim_k P_k$ in $\Fun(\catC^\op,\catV)$. Now that $\Fun(\catC^\op,\catV)$ is a reflective full subcategory of $\Mod(\catG)$, we have $\lim_k P_k \cong S'(\lim_k T'(P_k))$. Then by the left exactness of $S_0 \circ S'$ it follows that
	\[ S_0(\lim\nolimits_k P_k) \cong S_0\circ S'(\lim\nolimits_k T'(P_k)) \cong \lim\nolimits_k S_0 \circ S' (T'(P_k)) \cong \lim\nolimits_k S_0(P_k), \]
	which proves $S_0$ is left exact. Therefore we obtain the desired conclusion.
\end{proof}

As an application of the intrinsic characterization of Grothendieck enriched categories we obtained, we show that the property of being a Grothendieck enriched category is preserved by the change-of-base functors associated to monoidal functors that have strongly monoidal left adjoints.

\begin{proposition}\label{prop:monoidal_right_adjoint_preserve_Grothendieck_V-category}
	Let $\moncatV,\moncatW$ be Grothendieck cosmoi and $F\dashv G\colon \moncatV\to\moncatW$ a monoidal adjunction.	Suppose that $\moncatV$ satisfies the conditions (C1) and (C2). If a $\moncatW$-category $\catB$ is a Grothendieck $\moncatW$-category, then the $\moncatV$-category $G(\catB)$ is a Grothendieck $\moncatV$-category.
\end{proposition}

\begin{proof}
	From the proof of the necessity of \cref{thm:main_theorem}, the Grothendieck $\moncatW$-category $\catB$ is complete and cocomplete and its underlying category $\catB_0$ is a Grothendieck category.
	
	Let us check that the $\moncatV$-category $G(\catB)$ satisfies the conditions (i)--(v) in \cref{thm:main_theorem}. Since $\catB$ is complete and cocomplete, \cref{cor:monoidal_right_adjoint_preserve_completeness} yields that $G(\catB)$ is also complete and cocomplete. As is seen in \cref{prop:monoidal_right_ajoint_commute_with_underlying_functor}, there is an isomorphism $(G(\catB))_0 \cong \catB_0$ of categories. Thus the homomorphism theorem holds, and by \cref{lem:S_is_left_exact_when_S0_is_so} conical filtered colimits are left exact in $G(\catB)$. Moreover $(G(\catB))_0$ has a generating set of objects, and hence $G(\catB)$ has a $\moncatV$-generating set of objects by \cref{prop:generators_form_V-generators}. Therefore it follows from \cref{thm:main_theorem} that $G(\catB)$ is a Grothendieck $\moncatV$-category.
\end{proof}

\begin{example}\label{example:general_ver_of_Qcoh}
	Let $X,Y$ be schemes and $f\colon X\to Y$ a quasi-compact and quasi-separated morphism between them. We also suppose that the Grothendieck cosmos $\Qcoh(Y)$ satisfies the conditions (C1) and (C2); this is the case if $Y$ is projective by \cref{example:G-cosmos_of_quasi-coh}. Since $f$ is quasi-compact and quasi-separated, the pushout functor $f_*\colon \Mod(\mO_X) \to \Mod(\mO_Y)$ preserves the quasi-coherence, inducing a functor $f_*\colon \Qcoh(X)\to \Qcoh(Y)$. Thus we have an adjunction
	\[
	\begin{tikzpicture}
	\node at (-0.2,0) {$f^*:\Qcoh(Y)$};
	\node at (3.6,0) {$\Qcoh(X):f_*$,};
	\draw[->] (0.9,0.16) to (2.5,0.16);
	\draw[->] (2.5,-0.16) to (0.9,-0.16);
	\node at (1.7,0) {$\perp$};
	\end{tikzpicture}
	\]
	which turns out to be a monoidal adjunction. If $\Ch(\Qcoh(X))$ denotes the category of complexes of objects of $\Qcoh(X)$, then it is also a Grothendieck cosmos with the total tensor product and the total internal Hom (\cref{example:Ch(V)_satisfy_C1andC2}).
	
	If we view $\Qcoh(X)$ itself as the $\Qcoh(X)$-category $\catB=\Qcoh(X)$, then it is clearly a Grothendieck $\Qcoh(X)$-category. Hence, by \cref{prop:monoidal_right_adjoint_preserve_Grothendieck_V-category}, the $\Qcoh(Y)$-category $\catA\coloneqq f_*(\catB)$ is a Grothendieck $\Qcoh(Y)$-category. Recall that $\catA$ here is the $\Qcoh(Y)$-category whose objects are quasi-coherent sheaves on $X$ and whose Hom objects are $f_*\intHom_X^\qc(\sheafF,\sheafG) \in \Qcoh(Y)$.
	
	Furthermore, the monoidal adjunction $f^*\dashv f_*\colon \Qcoh(Y) \to \Qcoh(X)$ induces the monoidal adjunction $f^*\dashv f_*\colon \Ch(\Qcoh(Y)) \to \Ch(\Qcoh(X))$ between categories of complexes, and the Grothendieck cosmos $\Ch(\Qcoh(Y))$ also satisfies the conditions (C1) and (C2) by \cref{example:Ch(V)_satisfy_C1andC2}.
	
	Regarding $\Ch(\Qcoh(X))$ itself as a $\Ch(\Qcoh(X))$-category $\catB'=\Ch(\Qcoh(X))$, we also see that it is a Grothendieck $\Ch(\Qcoh(X))$-category, and hence by \cref{prop:monoidal_right_adjoint_preserve_Grothendieck_V-category} the $\Ch(\Qcoh(Y))$-category $\catA'\coloneqq f_*(\catB')$ is a Grothendieck $\Ch(\Qcoh(Y))$-category. Note that $\catA'$ has complexes of quasi-coherent sheaves on $X$ as objects and $f_* \cpx{\intHom_X^\qc}(\cpx{\sheafF},\cpx{\sheafG}) \in \Ch(\Qcoh(Y))$ as Hom objects.
\end{example}

\begin{example}\label{example:dg_cat_on_qcqs_scheme_over_R}
	As a special case of \cref{example:general_ver_of_Qcoh}, consider a quasi-compact and quasi-separated scheme $f\colon X \to \Spec(R)$ over a commutative ring $R$ (for example, separated schemes of finite type over a field). Via the equivalence $\Mod(R)\simeq \Qcoh(\Spec(R))$ of categories, the pushout functor $f_*$ is naturally isomorphic to the global section functor $\Gamma(X,\mplaceholder)\colon \Qcoh(X)\to \Mod(R)$. Then we can verify $f_*\intHom_X^\qc(\sheafF,\sheafG) \cong \Hom_{\mO_X}(\sheafF,\sheafG)$.
	
	Define a $\Ch(R)$-category (i.e., a dg category over $R$) $\catA$ as follows: objects of $\catA$ are complexes of quasi-coherent sheaves on $X$ and Hom complexes of $\catA$ are $\catA(\cpx{\sheafF},\cpx{\sheafG})\in \Ch(R)$ such that
	\begin{gather*}
		\catA(\cpx{\sheafF},\cpx{\sheafG})^n = \prod_{i \in \Z} \Hom_{\mO_X}(\sheafF^i,\sheafG^{i+n}), \\
		d^n(h) = \{ d_{\cpx{\sheafG}}^{i+n} \circ h^i - (-1)^n h^{i+1} \circ d_{\cpx{\sheafF}}^i \}_{i \in \Z}.
	\end{gather*}
	In other words, 
	\[ \catA(\cpx{\sheafF},\cpx{\sheafG}) \cong f_*(\cpx{\intHom_X^\qc}(\cpx{\sheafF},\cpx{\sheafG})). \]
	Then $\catA\cong f_*(\Ch(\Qcoh(X)))$ holds and it follows from \cref{prop:monoidal_right_adjoint_preserve_Grothendieck_V-category} that $\catA$ is a Grothendieck dg category over $R$.
\end{example}

We can prove that the adjoint functor theorem holds for Grothendieck $\moncatV$-categories. This indicates the usefulness of Grothendieck $\moncatV$-categories.

\begin{proposition}\label{prop:adjoint_functor_theorem}
	Let $\moncatV$ be a locally finitely presentable base and $F\colon \catA\to\catB$ a $\moncatV$-functor. Suppose $\catA$ is a Grothendieck $\moncatV$-category. If it preserves all small colimits, then $F$ has a right adjoint.
\end{proposition}

\begin{proof}
	From \cref{def:Grothendieck_V-category}, we observe that the $\moncatV$-functor $S\circ\yoneda\colon \colon \catC \to \catA$ is dense in the sense of \cite[\S5.1]{Kelly:1982Basic} and $\catA$ has a small dense full subcategory (\cite[\S5.3]{Kelly:1982Basic}). Therefore it follows from \cite[Thm.\ 5.33]{Kelly:1982Basic} that $F$ has a right adjoint.
\end{proof}

\begin{proposition}\label{prop:AFT_for_C1C2GrothCosmos}
	Let $\moncatV$ be a Grothendieck cosmos with the conditions (C1) and (C2) and $F\colon \catA\to\catB$ a $\moncatV$-functor.
	\begin{enumerate}
		\item Suppose $\catA$ is a Grothendieck $\moncatV$-category. If it preserves all small conical colimits, then $F$ has a right adjoint.
		\item Suppose $\catA$ is a Grothendieck $\moncatV$-category and $\catB$ is cotensored. If the underlying ordinary functor $F_0$ is cocontinuous, then $F$ has a right adjoint.
	\end{enumerate}
\end{proposition}

\begin{proof}
	Note that $\moncatV$ is a locally finitely presentable base by \cref{prop:C1_C2_implies_lfp_base}. Since $\moncatV$ has a generating set of dualizable  objects, \cref{prop:equivalent_condition_of_generators_in_abelian_cat} shows that any object $X\in \moncatV$ has an exact sequence of the form
	\begin{alignat}{2}
		\bigoplus_i P_i \to \bigoplus_j P_j \to X \to 0&  & \qquad & \tag{$\vartriangle$}
	\end{alignat}
	where $P_i, P_j$ are dualizable.
	
	(i) It follows from \cref{prop:end_is_conical_limit_of_cotensor,prop:limit_is_end_of_cotensor} that all colimits can be written as conical colimits of tensor products. Hence we only need to show that $F$ preserves tensor products, from which \cref{prop:adjoint_functor_theorem} deduces that $F$ has a right adjoint.
	
	For $X\in \moncatV$ and $A\in \catA$, consider the tensor product $X\tensor A$. Since the $\moncatV$-functor $\mplaceholder\tensor A\colon \catV\to\catA$ preserves colimits, ($\vartriangle$) induces the exact sequence
	\[
	\bigoplus\nolimits_i (P_i \tensor A) \to \bigoplus\nolimits_j (P_j \tensor A) \to X\tensor A \to 0.
	\]
	Using the assumption that $F$ preserves conical colimits and the fact that $P_i\tensor A, P_j \tensor A$ are absolute colimits, we obtain the exact sequence
	\[
	\bigoplus\nolimits_i (P_i \tensor F(A)) \to \bigoplus\nolimits_j (P_j \tensor F(A)) \to F(X\tensor A) \to 0.
	\]
	Thus the left exactness of $\catB(\mplaceholder,B)$ gives the exact sequence
	\[
	0 \to \catB(F(X\tensor A),B) \to \catB\left(\bigoplus\nolimits_j (P_j \tensor F(A)),B\right) \to \catB\left(\bigoplus\nolimits_i (P_i \tensor F(A)), B\right).
	\]
	On the other hand, by the left exactness of $\catV(\mplaceholder,\catB(F(A),B))$, we get from ($\vartriangle$) the exact sequence
	\[
	0 \to \catV(X,\catB(F(A),B)) \to \catV\left(\bigoplus\nolimits_i P_i,\catB(F(A),B)\right) \to \catV\left(\bigoplus\nolimits_j P_j, \catB(F(A),B)\right).
	\]
	Then, considering isomorphisms
	\[ \catB\left(\bigoplus\nolimits_i (P_i \tensor F(A)), B\right) \cong \prod_i \catB(P_i\tensor F(A),B) \cong \prod_i \catV(P_i,\catB(F(A),B)) \cong \catV\left(\bigoplus\nolimits_i P_i,\catB(F(A),B)\right), \]
	we have $\catB(F(X\tensor A),B) \cong \catV(X,\catB(F(A),B))$. Therefore $F(X\tensor A)$ forms the tensor product $X\tensor F(A)$, and hence $F$ preserves the tensor product $X\tensor A$.
	
	(ii) Because $\catB$ is cotensored, \cref{prop:tensored_and_ordinary_limit_implies_conical_limit} shows that ordinary colimits exsisting in $\catB_0$ become conical colimits in $\catB$. The cocontinuity of $F_0$ implies that $F$ preserves conical colimits, so the assertion follows from (i).
\end{proof}

\begin{remark}
	There is the fact that an arbitrary continuous functor from a Grothendieck category $\catA$ to $\Set$ is representable, from which the adjoint functor theorem for $\catA$ follows.
	The author expects that the similar result should hold for Grothendieck enriched categories as well, though he has not found a proof yet.
\end{remark}

\begin{remark}\label{remark:the_future_task}
	One of the motivations of this work was to establish the derived-dg version of the theory of Grothendieck categories.
	In this paper we succeeded in formulating an analogue of Grothendieck categories in the $2$-category $\dgCat = \ChCat$ of dg categories,
	where $\Ch$ is the Grothendieck cosmos of complexes of abelian groups. On the other hand, in view of the relations to triangulated categories and stable infinity categories,
	it seems more appropriate to work in the localization $\HodgCat$ of $\dgCat$ with respect
	to quasi-equivalences and use derived dg categories instead of the dg categories of dg modules.	Unfortunately our methods do not apply to $\HodgCat$ immediately, for the reason that it does not admit a canonical nontrivial $2$-categorical structure.
	We had to use such a structure to define the notion of Grothendieck categories in terms of adjunctions. Recall that we used adjunctions induced by Kan extensions in order to obtain the nerve-and-realization adjunction,
	which plays an important role in the proof of \cref{thm:main_theorem}. Solving this issue is a future task.
\end{remark}

\begin{remark}
	There are a few proposed definitions for the $2$-categorical structures on $\HodgCat$. For example, one of them identifies morphisms of $\HodgCat$ with \emph{right quasi-representable} bimodules (\cite{Keller:2006,Toen:2007}). Then morphisms of bimodules in the derived category play the role of $2$-morphisms between morphisms in $\HodgCat$. Concerning this $2$-categorical structure, \cite{Genovese:2017} defines adjunctions of morphisms in $\HodgCat$. Also \cite{Lowen-RamosGonzalez:2020} introduces the notion of a dg Bousfield localization of pretriangulated dg categories, which is equivalently defined as a quasi-fully faithful dg functor that has a left adjoint as a quasi-functor. 
	The author does not have an explicit explanation as to the relations of the results of this paper to these work.
\end{remark}

\end{document}